\documentclass[twoside,leqno,10pt, A4]{amsart}
\usepackage{amsfonts}
\usepackage{amsmath}
\usepackage{amscd}
\usepackage{amssymb}
\usepackage{amsthm}
\usepackage{amsrefs}
\usepackage{latexsym}
\usepackage{mathrsfs}
\usepackage{bbm}
\usepackage{amscd}
\usepackage{amssymb}
\usepackage{amsthm}
\usepackage{amsrefs}
\usepackage{latexsym}
\usepackage{mathrsfs}
\usepackage{bbm}
\usepackage{enumerate}
\usepackage{graphicx}
\usepackage{color}
\begin{document}

\newtheorem{theorem}[subsection]{Theorem}
\newtheorem{proposition}[subsection]{Proposition}
\newtheorem{lemma}[subsection]{Lemma}
\newtheorem{corollary}[subsection]{Corollary}
\newtheorem{conjecture}[subsection]{Conjecture}
\newtheorem{prop}[subsection]{Proposition}
\newtheorem{defin}[subsection]{Definition}

\numberwithin{equation}{section}
\newcommand{\mr}{\ensuremath{\mathbb R}}
\newcommand{\mc}{\ensuremath{\mathbb C}}
\newcommand{\dif}{\mathrm{d}}
\newcommand{\intz}{\mathbb{Z}}
\newcommand{\ratq}{\mathbb{Q}}
\newcommand{\natn}{\mathbb{N}}
\newcommand{\comc}{\mathbb{C}}
\newcommand{\rear}{\mathbb{R}}
\newcommand{\prip}{\mathbb{P}}
\newcommand{\uph}{\mathbb{H}}
\newcommand{\fief}{\mathbb{F}}
\newcommand{\majorarc}{\mathfrak{M}}
\newcommand{\minorarc}{\mathfrak{m}}
\newcommand{\sings}{\mathfrak{S}}
\newcommand{\fA}{\ensuremath{\mathfrak A}}
\newcommand{\mn}{\ensuremath{\mathbb N}}
\newcommand{\mq}{\ensuremath{\mathbb Q}}
\newcommand{\half}{\tfrac{1}{2}}
\newcommand{\f}{f\times \chi}
\newcommand{\summ}{\mathop{{\sum}^{\star}}}
\newcommand{\chiq}{\chi \bmod q}
\newcommand{\chidb}{\chi \bmod db}
\newcommand{\chid}{\chi \bmod d}
\newcommand{\sym}{\text{sym}^2}
\newcommand{\hhalf}{\tfrac{1}{2}}
\newcommand{\sumstar}{\sideset{}{^*}\sum}
\newcommand{\sumprime}{\sideset{}{'}\sum}
\newcommand{\sumprimeprime}{\sideset{}{''}\sum}
\newcommand{\sumflat}{\sideset{}{^\flat}\sum}
\newcommand{\shortmod}{\ensuremath{\negthickspace \negthickspace \negthickspace \pmod}}
\newcommand{\V}{V\left(\frac{nm}{q^2}\right)}
\newcommand{\sumi}{\mathop{{\sum}^{\dagger}}}
\newcommand{\mz}{\ensuremath{\mathbb Z}}
\newcommand{\leg}[2]{\left(\frac{#1}{#2}\right)}
\newcommand{\muK}{\mu_{\omega}}
\newcommand{\thalf}{\tfrac12}
\newcommand{\lp}{\left(}
\newcommand{\rp}{\right)}
\newcommand{\Lam}{\Lambda_{[i]}}
\newcommand{\lam}{\lambda}
\newcommand{\af}{\mathfrak{a}}
\newcommand{\sw}{S_{[i]}(X,Y;\Phi,\Psi)}
\newcommand{\lz}{\left(}
\newcommand{\pz}{\right)}
\newcommand{\bfrac}[2]{\lz\frac{#1}{#2}\pz}
\newcommand{\odd}{\mathrm{\ primary}}
\newcommand{\even}{\text{ even}}
\newcommand{\res}{\mathrm{Res}}
\newcommand{\sumn}{\sumstar_{(c,1+i)=1}  w\left( \frac {N(c)}X \right)}
\newcommand{\lab}{\left|}
\newcommand{\rab}{\right|}
\newcommand{\Go}{\Gamma_{o}}
\newcommand{\Ge}{\Gamma_{e}}
\newcommand{\M}{\widehat}
\def\su#1{\sum_{\substack{#1}}}

\theoremstyle{plain}
\newtheorem{conj}{Conjecture}
\newtheorem{remark}[subsection]{Remark}

\newcommand{\pfrac}[2]{\left(\frac{#1}{#2}\right)}
\newcommand{\pmfrac}[2]{\left(\mfrac{#1}{#2}\right)}
\newcommand{\ptfrac}[2]{\left(\tfrac{#1}{#2}\right)}
\newcommand{\pMatrix}[4]{\left(\begin{matrix}#1 & #2 \\ #3 & #4\end{matrix}\right)}
\newcommand{\ppMatrix}[4]{\left(\!\pMatrix{#1}{#2}{#3}{#4}\!\right)}
\renewcommand{\pmatrix}[4]{\left(\begin{smallmatrix}#1 & #2 \\ #3 & #4\end{smallmatrix}\right)}
\def\en{{\mathbf{\,e}}_n}

\newcommand{\ppmod}[1]{\hspace{-0.15cm}\pmod{#1}}
\newcommand{\ccom}[1]{{\color{red}{Chantal: #1}} }
\newcommand{\acom}[1]{{\color{blue}{Alia: #1}} }
\newcommand{\alexcom}[1]{{\color{green}{Alex: #1}} }
\newcommand{\hcom}[1]{{\color{brown}{Hua: #1}} }

\makeatletter
\def\widebreve{\mathpalette\wide@breve}
\def\wide@breve#1#2{\sbox\z@{$#1#2$}%
     \mathop{\vbox{\m@th\ialign{##\crcr
\kern0.08em\brevefill#1{0.8\wd\z@}\crcr\noalign{\nointerlineskip}%
                    $\hss#1#2\hss$\crcr}}}\limits}
\def\brevefill#1#2{$\m@th\sbox\tw@{$#1($}%
  \hss\resizebox{#2}{\wd\tw@}{\rotatebox[origin=c]{90}{\upshape(}}\hss$}
\makeatletter

\title[Ratios conjecture of quartic $L$-functions of prime moduli]{Ratios conjecture of quartic $L$-functions of prime moduli}

\author[P. Gao]{Peng Gao}
\address{School of Mathematical Sciences, Beihang University, Beijing 100191, China}
\email{penggao@buaa.edu.cn}

\author[L. Zhao]{Liangyi Zhao}
\address{School of Mathematics and Statistics, University of New South Wales, Sydney NSW 2052, Australia}
\email{l.zhao@unsw.edu.au}

\begin{abstract}
We apply the method of multiple Dirichlet series to develop $L$-functions ratios conjecture with one shift in both the numerator
and denominator in certain ranges for the family of quartic Hecke $L$-functions of prime moduli over the Gaussian field under the generalized Riemann
hypothesis. As consequences, we evaluate asymptotically the first moment of central values as well as the one-level density of the same family of $L$-functions.
\end{abstract}

\maketitle



\section{Introduction}\label{sec 1}

  The $L$-functions ratios conjecture originates from the work of D. W. Farmer \cite{Farmer93} on the shifted moments of the Riemann zeta function and is formulated for general $L$-functions by J. B. Conrey, D. W. Farmer and M. R. Zirnbauer in \cite[Section 5]{CFZ} to make predictions on
the asymptotic behaviors of the sum of ratios of products of shifted $L$-functions. The conjecture has been applied extensively to investigate subjects including the density conjecture of N. Katz and P. Sarnak \cites{KS1, K&S} on the distribution of zeros near the central point of a family of $L$-functions and the mollified moments of $L$-functions, etc. \newline

 In \cite{BFK21}, H. M. Bui, A. Florea and J. P. Keating  established the ratios conjecture on quadratic $L$-functions over function fields for
certain ranges of parameters. On the number fields side, the first result concerning ratios conjecture is given by M. \v Cech \cite{Cech1}, who applied the powerful method of multiple Dirichlet series to study families of quadratic Dirichlet $L$-functions under the assumption of the generalized Riemann hypothesis (GRH).  Work of this kind has been carried out in the setting of function fields by V. Y. Wang \cite{VYWang}. \newline

  In \cite{G&Zhao17}, the authors developed the ratios conjecture with one shift in the numerator
and denominator in certain ranges for families of cubic Hecke $L$-functions of prime moduli over the Eisenstein field $\mq(\frac {-1+\sqrt{3}i}{2})$. Unlike in \cite{BFK21} and \cite{Cech1}, the families of $L$-functions studied in \cite{G&Zhao17} have the feature that their conductors form sparse sets (sets of density zero) in the ring of integers. \newline

 The purpose of this paper is to extend the results in \cite{G&Zhao17} to families of quartic Hecke $L$-functions of prime moduli over the Gaussian field $\mq(i)$. Before stating our result, we introduce some notation first. We fix $K=\mq(i)$ for the Gaussian field throughout the paper to denote $\mathcal{O}_K, U_K$ for the ring of integers and the group of units in $\mathcal{O}_K$, respectively. The symbol $\varpi$ is reserved for a prime in $\mathcal{O}_K$, which means that the ideal $(\varpi)$ generated by $\varpi$ is a prime ideal. Let $\chi$ be a Hecke character of $K$ and we say that $\chi$ is of trivial infinite type if its component at infinite places of $K$ is trivial. We write $L(s,\chi)$ for the $L$-function associated to $\chi$ and we denote $L^{(c)}(s, \chi)$ for the Euler product defining $L(s, \chi)$ but omitting those primes dividing $c$. Further, let $\zeta_{K}(s)$, $\zeta(s)$ denote the Dedekind zeta function of $K$ and the Riemann zeta function, respectively.   \newline
 
Set $\lambda := 1+i$ so that the ideal $(1+i)$ is the only prime ideal that lies above the rational ideal $(2) \in \mz$.  Note that (see the paragraph above \cite[Lemma 8.2.1]{BEW}) every ideal in $\mathcal{O}_K$ co-prime to $\lambda$ has a unique generator congruent to $1$ modulo $\lambda^3$ which is called primary. Let $\chi_{j, \varpi}$ be a quartic Hecke symbol defined in Section \ref{sec2.4} where we fix $j=4$ throughout the paper.  It will be shown later in Section \ref{sec2.4} that $\chi_{j, \varpi}$ is a Hecke character of trivial infinite type when $\varpi \equiv 1 \pmod {16}$.  Here we note that the congruence condition $m \equiv 1 \pmod {16}$ can be determined using ray class group characters.  Recall that for any integral ideal $\mathfrak{m} \in O_K$, the ray class group $h_{\mathfrak{m} }$ is defined to be $I_{\mathfrak{m} }/P_{\mathfrak{m} }$, where $I_{\mathfrak{m} } = \{
\mathcal{A} \in I : (\mathcal{A}, \mathfrak{m} ) = 1 \}$ and $P_{\mathfrak{m} } = \{(a) \in P : a \equiv 1 \bmod \mathfrak{m}  \}$ with $I$ and $P$ denoting the group of
fractional ideals in $K$ and the subgroup of principal ideals, respectively. \newline

 We write $N(n)$  for the norm of any $n \in K$, respectively.  We denote $\mu_K$ for the M\"obius function on $\mathcal O_K$ and let $\Lambda_K(n)$ stand for the von Mangoldt function on $\mathcal{O}_K$ given by
\begin{align*}
    \Lambda_K(n)=\begin{cases}
   \log N(\varpi) \qquad & n=\varpi^k, \text{$\varpi$ prime}, k \geq 1, \\
     0 \qquad & \text{otherwise}.
    \end{cases}
\end{align*}  

   Our main result investigates under GRH the ratios conjecture with one shift in the numerator
and denominator for the family of quartic Hecke $L$-functions of prime moduli. As usual, in the rest of the paper, we use $\varepsilon$ to denote a small positive quantity which may not be the same at each occurrence.
\begin{theorem}
\label{Theorem for all characters}
	With the notation as above and assuming the truth of GRH, let $X$ be a large real number, $w(t)$ a non-negative Schwartz function with $\widehat w(s)$ being its Mellin transform.  Set
\begin{align}
\label{Nab}
			E(\alpha,\beta)=\max\left\{\frac 12, \ \frac 56-\Re(\alpha), \ 1-\Re(\beta),  \ 1-3\Re(\alpha)-2\Re(\beta), \ 1-\frac {13}{15}\Re(\alpha)-\frac {2}{15}\Re(\beta), \ \frac {12}{13}-\frac {11}{13}\Re(\alpha) \right\}.
\end{align}
  We also set $\delta(\alpha)=13/11$ when $\Re(\alpha) <0$ and $\delta(\alpha)=0$ otherwise. Then for $\Re(\alpha) > -1/11$ and $\Re(\beta)>\varepsilon$ such that $E(\alpha,\beta)<1$,
\begin{align}
\label{Asymptotic for ratios of all characters}
\begin{split}	
 \sumstar_{\substack{\varpi}} & \frac {\Lambda_{K}(\varpi)L(\tfrac 12+\alpha, \chi_{j, \varpi})}{L(\tfrac 12+\beta, \chi_{j, \varpi})}w \bfrac {N(\varpi)}X  \\
=&  \frac {\M w(1)X}{\# h_{(16)}} \frac{1-2^{-1/2-\beta}}{1-2^{-1/2-\alpha}}
\frac {\zeta^{(j)}_K(2+4\alpha)}{\zeta^{(j)}_K(2+3\alpha+\beta)} +O\lz(1+|\alpha|)^{\delta(\alpha)+\varepsilon}(1+|\beta|)^{\varepsilon} X^{E(\alpha,\beta)+\varepsilon}\pz,
\end{split}
\end{align}
 where the ``$*$'' on the sum over $\varpi$ means that the sum is restricted to prime elements $\varpi$ of $\mathcal{O}_K $ that are congruent to 1 modulo 16.
\end{theorem}

   Our result above is consistent with the prediction from the ratios conjecture on the left-hand side of \eqref{Asymptotic for ratios of all characters}, which can be derived following the recipe given in \cite[Section 5]{CFZ} except that (see also \cite{CS}) the ratios conjecture asserts that \eqref{Asymptotic for ratios of all characters} holds uniformly for $|\Re(\alpha)|< 1/4$, $(\log X)^{-1} \ll \Re(\beta) < 1/4$ and $\Im(\alpha)$, $\Im(\beta) \ll X^{1-\varepsilon}$ with an error term $O(X^{1/2+\varepsilon})$.  In contrast, Theorem \ref{Theorem for all characters} has the advantage that there is no constraint on imaginary parts of $\alpha$ and $\beta$. \newline

   Upon taking the limit $\beta \rightarrow \infty$ on both sides of \eqref{Asymptotic for ratios of all characters} and observing that in this process eliminates the factor involving $\beta$ in the error term, we readily obtain the following result on the first moment of quartic Hecke $L$-functions.
\begin{theorem}
\label{Thmfirstmoment}
		With the notation as above and assuming the truth of GRH, we have for $\Re(\alpha)>-1/11$ and any $\varepsilon>0$,
\begin{align}
\label{Asymptotic for first moment}
\begin{split}			
& \sumstar_{\substack{\varpi}}\Lambda_{K}(\varpi)L(\tfrac 12+\alpha, \chi_{j, \varpi}) w \bfrac {N(\varpi)}X  =  \frac {\M w(1)X}{\# h_{(16)}}
 \frac{\zeta^{(j)}_K(2+4\alpha)}{ 1-2^{-1/2-\alpha}} +O\lz(1+|\alpha|)^{\delta(\alpha)+\varepsilon}X^{E(\alpha)+\varepsilon}\pz,
\end{split}
\end{align}
  where
\begin{align}
\label{Ealpha}
\begin{split}
      &  E(\alpha):=\lim_{\beta \rightarrow \infty}E(\alpha,\beta)=\max\left\{\frac 12, \ \frac 56-\Re(\alpha), \ \frac {12}{13}-\frac {11}{13}\Re(\alpha)  \right\}.
\end{split}
\end{align} 	
\end{theorem}

 We remark here that setting $\alpha=0$ in \eqref{Asymptotic for first moment} gives an asymptotic formula for the first moment of central values of the family of quartic Hecke $L$-functions of prime moduli with an error term of size $O(X^{12/13+\varepsilon})$.  This lead to an improvement of the result on quartic Hecke $L$-functions given in \cite[Theorem 1.2]{G&Zhao2022-4}. \newline

  Similarly, we set $\alpha \rightarrow \infty$ on both sides of \eqref{Asymptotic for ratios of all characters} and drop the factor involving with $\alpha$ in the error term to obtain the following formula for the negative first moment of quartic Hecke $L$-functions.
\begin{theorem}
\label{Thmnegfirstmoment}
		With the notation as above and assuming the truth of GRH. We have for $\Re(\beta)>0$ and any $\varepsilon>0$,
\begin{align*}
\begin{split}			
& \sumstar_{\substack{\varpi}}\frac {\Lambda_{K}(\varpi)}{L(\tfrac 12+\beta, \chi_{j, \varpi})} w \bfrac {N(\varpi)}X  =  \frac {\M w(1)X}{\# h_{(16)}}(1-2^{-1/2-\beta})
 +O\lz (1+|\beta|)^{\varepsilon} X^{\max\left\{1/2, 1-\Re(\beta) \right\}+\varepsilon}\pz.
\end{split}
\end{align*}
\end{theorem}

  We may also differentiate with respect to $\alpha$ in \eqref{Asymptotic for ratios of all characters} and then set $\alpha=\beta=r$ to obtain an asymptotic formula for the smoothed first moment of $L'(\frac{1}{2}+r,\chi_{j, \varpi})/L(\frac{1}{2}+r,\chi_{j, \varpi})$ averaged over $\varpi \equiv 1 \pmod {16}$.
\begin{theorem}
\label{Theorem for log derivatives}
	With the notation as above and assuming the truth of GRH, we have for $0<\varepsilon< \Re(r)<1/2$,
\begin{align*}
\begin{split}
		& \sumstar_{\substack{\varpi}}\frac {\Lambda_{K}(\varpi)L'(\tfrac 12+r, \chi_{j, \varpi})}{L(\tfrac 12+r, \chi_{j, \varpi})}w \bfrac {N(\varpi)}X
			= \frac {\M w(1)X }{\# h_{(16)}}\lz\frac{(\zeta^{(j)}_K(2+4r))'}{\zeta^{(j)}_K(2+4r)}-\frac{\log 2 }{2^{1/2+r}-1}\pz+O((1+|r|)^{\varepsilon}X^{1-\Re(r)+\varepsilon}).
\end{split}
\end{align*}
\end{theorem}
	
 Now, Theorem \ref{Theorem for log derivatives} makes it possible to compute the one-level density of low-lying zeros of the corresponding families of quartic Hecke $L$-functions.  This can be done following the approach in the proof of \cite[Corollary 1.5]{Cech1} using \cite[Theorem 1.4]{Cech1} for the case of quadratic Dirichlet $L$-functions.  Let $h(x)$ be an even Schwartz function such that its Fourier transform is supported in the interval $[-a,a]$ for some $a>0$.
We define the one-level density of the family of $L$-functions considered in this paper by
\begin{align*}
\begin{split}
		D(X;h)=\frac{1}{F(X)}\sumstar_{\substack{ \varpi \odd}}\Lambda_K(\varpi)w \bfrac {N(\varpi)}X\sum_{\gamma_{\varpi, n}}h\bfrac{\gamma_{\varpi, n}\log X}{2\pi},
\end{split}
\end{align*}
   where $\gamma_{\varpi, n}$ runs over the imaginary parts of the non-trivial zeros of $L(s,\chi_{j, \varpi})$, and
\begin{align}
\label{Size of family}
\begin{split}
			F(X)&=\sumstar_{\substack{\varpi \odd}}\Lambda_K(\varpi) w \bfrac {N(\varpi)}X.
\end{split}
\end{align}
	
   With the aid of Theorem \ref{Theorem for log derivatives}, our next result evaluates $D(X;h)$ asymptotically.
\begin{theorem}
\label{Theorem one-level density}
	With the notation as above and assuming the truth of GRH, for any function $w(t)$ that is non-negative and compactly supported on the set of positive real numbers, we have
\begin{align*}
\begin{split}
		D(X;h)= \frac {2\widehat h(1)}{F(X)\log X} & \sumstar_{\substack{\varpi \odd}}\Lambda_K(\varpi) w \bfrac {N(\varpi)}X\log N(\varpi) \\
& \hspace*{-1cm} +\frac{1}{F(X)}\frac {2 \M w(1) }{\# h_{(16)}} \cdot \frac{X}{\log X}\int\limits_{-\infty}^{\infty}h(u)\lz\frac{(\zeta^{(j)}_K(2+ \frac {8\pi iu}{\log X}))'}{\zeta^{(j)}_K(2+\frac {8\pi iu}{\log X})}-\frac{\log 2 }{2^{\frac 12+\frac {2\pi iu}{\log X}}-1}\pz \dif u \\
& \hspace*{-1cm} +\frac 1{\log X}\int\limits^{\infty}_{-\infty}h\lz u \pz \lz \frac {\Gamma'}{\Gamma} \left( \frac 12+\frac {2 \pi iu}{\log X} \right) +\frac {\Gamma'}{\Gamma} \left( \frac 12-\frac {2 \pi iu}{\log X} \right) \pz \dif u+\frac {2\widehat h(1)}{\log X} \log \frac {|D_K|}{4\pi^2} +O(X^{(1+a)/2+\varepsilon}).
\end{split}
\end{align*}
  In particular, when $a<1$, we have
\begin{align*}
\begin{split}
		D(X;h)= \int\limits^{\infty}_{-\infty}h(x) \dif x  +\frac{2\widehat h(1)}{\log X} & \lz \frac{2(\zeta^{(j)}_K(2))'}{\zeta^{(j)}_K(2)}-\frac{2\log 2 }{2^{1/2}-1} + \frac {1}{\widehat w(1)} \int\limits^{\infty}_0 w(u)\log u \  \dif u +2\frac {\Gamma'}{\Gamma} \left( \frac 12 \right) +\log \frac {|D_K|}{4\pi^2} \pz \\
		& +O\left( \frac 1{(\log X)^2} \right).
\end{split}
\end{align*}
\end{theorem}

  Our considerations above apply also to quartic Dirichlet $L$-functions. To be more precise, we reserve the letter $p$ for a rational prime.  Let $\Lambda(n)$ denote the usual von Mangoldt function on $\mz$ and $\chi_0$ the principal Dirichlet character. With minor modifications, the proofs of our results on quartic
Hecke $L$-functions carry over to the family of quartic Dirichlet $L$-functions associated with characters whose conductors are rational primes that are norms of primes $\varpi \in \mathcal O_K$ with $\varpi \equiv 1 \pmod {16}$.  The following result gives the ratios conjecture of this family of $L$-functions.
\begin{theorem}
\label{Theorem for all charactersQ}
	With the notation as above and assuming the truth of GRH, let $Q$ be a large real number and $\Phi(t)$ be a non-negative Schwartz function with Mellin transform $\widehat \Phi(s)$.  With $E(\alpha,\beta)$ defined in \eqref{Nab} and $\delta(\alpha)$ in the statement of Theorem \ref{Theorem for all characters}, we have,  for $\Re(\alpha) > -1/11$ and $\Re(\beta)>0$ such that $E(\alpha,\beta)<1$,
\begin{align}
\label{Asymptotic for ratios of all charactersQ}
\begin{split}	
 \sum_{\substack{ p=N(\varpi) \\ \varpi \equiv 1 \bmod {16}} }\ &
\sum_{\substack{\chi \bmod{p} \\ \chi^4 = \chi_0, \ \chi^2 \neq \chi_0}} \Lambda(p) \frac {L(\tfrac{1}{2}+\alpha, \chi)}{L(\tfrac{1}{2}+\beta, \chi)}  \Phi \leg{p}{Q}  \\
=&  \frac {2\M \Phi(1)Q}{\# h_{(16)}} \frac{1-2^{-1/2-\beta}}{1-2^{-1/2-\alpha}}
\frac {\zeta^{(j)}(2+4\alpha)}{\zeta^{(j)}(2+3\alpha+\beta)} +O\lz(1+|\alpha|)^{\delta(\alpha)/2+\varepsilon}(1+|\beta|)^{\varepsilon} Q^{E(\alpha,\beta)+\varepsilon}\pz.
\end{split}
\end{align}
\end{theorem}

  Again, we take $\beta \rightarrow \infty$ on both sides of \eqref{Asymptotic for ratios of all charactersQ} and arrive at the following result concerning the first moment of quartic Dirichlet $L$-functions.
\begin{theorem}
\label{ThmfirstmomentQ}
		With the notation as above and assuming the truth of GRH. Let $E(\alpha)$ be defined as in \eqref{Ealpha}. We have for $\Re(\alpha)>-1/11$ and any $\varepsilon>0$,
\begin{align*}
\begin{split}			
& \sum_{\substack{ p=N(\varpi) \\ \varpi \equiv 1 \bmod {16}} }\;
\sum_{\substack{\chi \bmod{p} \\ \chi^4 = \chi_0, \ \chi^2 \neq \chi_0}} \Lambda(p) L(\tfrac 12+\alpha, \chi) \Phi \leg{p}{Q} =  \frac {2 \M \Phi(1)Q}{\# h_{(16)}}
\frac{\zeta^{(j)}(2+4\alpha)}{1-2^{-1/2-\alpha}} +O\lz(1+|\alpha|)^{\delta(\alpha)/2+\varepsilon}Q^{E(\alpha)+\varepsilon}\pz.
\end{split}
\end{align*}	
\end{theorem}

   Similarly,  taking $\alpha \rightarrow \infty$ on both sides of \eqref{Asymptotic for ratios of all charactersQ} leads to the following negative first moment result.
\begin{theorem}
\label{ThmnegfirstmomentQ}
		With the notation as above and assuming the truth of GRH, we have for $\Re(\beta)>0$ and any $\varepsilon>0$,
\begin{align*}
\begin{split}			
& \sum_{\substack{ p=N(\varpi) \\ \varpi \equiv 1 \bmod {16}} }\;
\sum_{\substack{\chi \bmod{p} \\ \chi^4 = \chi_0, \ \chi^2 \neq \chi_0}} \frac {\Lambda(p)}{L(\tfrac 12+\beta, \chi)} \Phi \bfrac {p}Q  =  \frac {2 \M \Phi(1)Q}{\# h_{(16)}}(1-2^{-1/2-\beta})
 +O\lz (1+|\beta|)^{\varepsilon} Q^{\max\left\{1/2, 1-\Re(\beta) \right\}+\varepsilon}\pz.
\end{split}
\end{align*}
\end{theorem}

Now differentiating with respect to $\alpha$ in \eqref{Asymptotic for ratios of all charactersQ} and setting $\alpha=\beta=r$ leads to an asymptotic formula for the smoothed first moment of the logarithmic derivative of the $L$-functions under our consideration.
 \begin{theorem}
\label{Theorem for log derivativesQ}
	With the notation as above and assuming the truth of GRH, we have for $0<\varepsilon< \Re(r)<1/2$,
\begin{align*}
\begin{split}
		& \sum_{\substack{ p=N(\varpi) \\ \varpi \equiv 1 \bmod {16}} }\;
\sum_{\substack{\chi \bmod{p} \\ \chi^4 = \chi_0, \chi^2 \neq \chi_0}} \frac {\Lambda(p)L'(\tfrac 12+r, \chi)}{L(\tfrac 12+r, \chi)}\Phi \bfrac {p}Q
			= \frac {2 \M \Phi(1)Q }{\# h_{(16)}}\lz\frac{(\zeta^{(j)}(2+4r))'}{\zeta^{(j)}(2+4r)}-\frac{\log 2 }{2^{1/2+r}-1}\pz+O((1+|r|)^{\varepsilon}Q^{1-\Re(r)+\varepsilon}).
\end{split}
\end{align*}
\end{theorem}

Moreover, a one-level density of low-lying zeros of this families of $L$-functions can be derived from Theorem \ref{Theorem for log derivativesQ}.  With $h(x)$ as in Theorem~\ref{Theorem one-level density}, set
\begin{align*}
\begin{split}
		D_{\mq}(Q;h)=\frac{1}{F_{\mq}(Q)}\sum_{\substack{ p=N(\varpi) \\ \varpi \equiv 1 \bmod {16}} }\;
\sum_{\substack{\chi \bmod{p} \\ \chi^4 = \chi_0, \chi^2 \neq \chi_0}} \Lambda(p) \Phi \bfrac {p}Q\sum_{\gamma_{\chi, n}}h\bfrac{\gamma_{\chi, n}\log Q}{2\pi},
\end{split}
\end{align*}
   where $\gamma_{\chi, n}$ runs over the imaginary parts of the non-trivial zeros of $L(s,\chi)$ and
\begin{align*}
\begin{split}
		F_{\mq}(Q)&=\sum_{\substack{ p=N(\varpi) \\ \varpi \equiv 1 \bmod {16}} }\;
\sum_{\substack{\chi \bmod{p} \\ \chi^4 = \chi_0, \chi^2 \neq \chi_0}} \Lambda(p) \Phi \bfrac {p}Q.
\end{split}
\end{align*}

The final result of this section with the following asymptotic evaluation of $D_{\mq}(X;h)$.
\begin{theorem}
\label{Theorem one-level densityQ}
	With the notation as above and assuming the truth of GRH, for any function $\Phi(t)$ that is non-negative and compactly supported on the set of positive real numbers, we have
\begin{align*}
\begin{split}
		D_{\mq}(Q;h)=\frac {2\widehat h(1)}{F_{\mq}(Q)\log Q} & \sum_{\substack{ p=N(\varpi) \\ \varpi \equiv 1 \bmod {16}} }\;
\sum_{\substack{\chi \bmod{p} \\ \chi^4 = \chi_0, \chi^2 \neq \chi_0}} \Lambda(p) \Phi \bfrac {p}Q \log p \\
& \hspace*{-1cm} +\frac{1}{F_{\mq}(Q)}\frac {4 \M w(1) }{ \# h_{(16)}} \cdot \frac{Q}{\log Q}\int\limits_{-\infty}^{\infty}h(u)\lz\frac{(\zeta^{(j)}(2+ \frac {8\pi iu}{\log Q}))'}{\zeta^{(j)}(2+\frac {8\pi iu}{\log Q})}-\frac{\log 2 }{2^{1/2+2\pi iu/\log Q}-1}\pz \dif u \\
& \hspace*{-1cm} +\frac 1{2\log Q}\int\limits^{\infty}_{-\infty}h\lz u \pz \lz \frac {\Gamma'}{\Gamma} \left( \frac 14+\frac {\pi iu}{\log Q} \right) +\frac {\Gamma'}{\Gamma} \left(\frac 14-\frac {\pi iu}{\log Q} \right) \pz \dif u+\frac {2\widehat h(1)}{\log Q} \log \frac {1}{\pi} +O(Q^{(1+a)/2+\varepsilon}).
\end{split}
\end{align*}
  In particular, when $a<1$, we have
\begin{align*}
\begin{split}
D_{\mq}(Q;h)= \int\limits^{\infty}_{-\infty}h(x) \dif x +\frac{2\widehat h(1)}{\log Q} & \lz \frac{2(\zeta^{(j)}(2))'}{\zeta^{(j)}(2)}-\frac{2\log 2 }{2^{1/2}-1} + \frac {1}{\widehat w(1)} \left( \int\limits^{\infty}_0w(u)\log u \dif u \right)+\frac {\Gamma'}{\Gamma}\left( \frac 14 \right) +\log \frac {1}{\pi} \pz  \\
& +O\left(\frac 1{(\log Q)^2}\right).
\end{split}
\end{align*}
\end{theorem}

  The proofs of Theorem \ref{Theorem for all characters}--\ref{Theorem one-level densityQ} use ideas similar to those applied to the studies on the cubic characters in \cite{G&Zhao17}.  We shall use the method of multiple Dirichlet series to establish sufficient analytical properties of the series involved to derive the desired results. However, the contrast will become saliently evident that the proofs for these quartic characters require considerably more efforts than those of the cubic ones. The main reason is that the process needs one to estimate a Dirichlet series formed from quartic Gauss sums given by 
\begin{align}
\label{h}
   H(r,s;\psi) :=\sum_{\substack{\varpi \text{ primary} \\ (\varpi,r)=1}}\frac {\Lambda_K(\varpi)\psi(\varpi)g_{K,j}(r,\varpi)}{N(\varpi)^s}, 
\end{align}
  where $g_{K,j}(r,\varpi)$ is a quartic Gauss sum defined in Section \ref{sec2.4}, $\psi$ is any ray class character modulo $16$, and $r$ is any primary element in $\mathcal O_K$. Understanding $H(r,s;\psi)$ is ultimately tied to estimating
\begin{align}
\label{Fadef}
F_a(z,r,\psi) := \sum_{\substack{b \odd \\ (r,b)=1, a \mid b\\ N(b) \le z}} \tilde{g}_\psi(r,b), 
\end{align} 
   where we write for simplicity, 
\begin{align}
\label{gtildedef}
\tilde g_\psi (r,c) :=& g_{K,j}(r,c) \psi(c) N(c)^{-1/2}. 
\end{align}
  Here $a$ is a primary square-free element in $\mathcal O_K$, and we say an element $n \in \mathcal O_K$ is square-free if it is not divisible by any prime power. \newline

The study of $F_a(z,r,\psi)$ then is reduced to developing of analytical properties of Dirichlet series of the form
\begin{align}
\label{twistedseries}
 \su{b \odd }  \frac {\psi(b)  g_{K,j}(\nu^2r,b)}{N(b)^{s}},
\end{align}
  where $\nu, r$ are primary elements in $\mathcal O_K$ with $(\nu, r)=1$ and $\nu$ square-free.  \newline

   It follows from T. Kubota's theory \cites{Kub1,Kub2} of metaplectic
Eisenstein series on the $4$-fold cover of $\operatorname{GL}_2$ over the Gaussian field that the series in \eqref{twistedseries} appear naturally as coefficients of the Fourier expansion of related Eisenstein series. We therefore devote Section \ref{sec: MES} in this paper to develop results sufficient for our purpose on these Eisenstein series. In fact, most of the analytical properties needed in this paper for the series given in \eqref{twistedseries} have already been established by S. J. Patterson in \cite[p. 200, Lemma]{P}, except for the estimation of the residues of the possible simple pole at $s=5/4$ of the series. This is in contrast to the cubic case, where the corresponding residues were given in closed form by Patterson in \cite{Patterson77}. For the quartic case, only partial information on these residues are given by T. Suzuki \cite{Suz1} for the case $\psi$ being the principal character. For our purpose, we need to understand the situation when a general $\psi$ is involved (the twisted case). More precisely, we aim to show that contribution from the $\nu$-aspect to the residue of the series given in \eqref{twistedseries} at $5/4$ is $\ll N(v)^{-1/4+\varepsilon}$, a result we shall achieve in Proposition \ref{resrelation} by generalizing the treatments in \cite{Suz1}. This result is also analogous to \cite[Proposition 8.2, Corollary 8.4]{DDHL}. \newline

With the aid of Proposition \ref{resrelation}, we proceed in Section \ref{section6} to prove the next result for $F_a(z,r,\psi)$.
\begin{prop}
\label{prop: Fbound}
 With the notation as above, suppose $a \in \mathcal O_K$ is square-free with $(a,r)=1$ and $\psi$ is a Dirichlet character on $\mathcal O_K$.  Then for $z$ satisfying $N(a) \leq z^{1/4}N(r)^{-1/8}$ and any $\varepsilon>0$, we have
\begin{align} \label{Bound4F}
\begin{split}
F_a(z,r,\psi) &\ll z^{3/4+\varepsilon} N(a )^{-1/2+\varepsilon} N (r)^{1/8+\varepsilon}.
\end{split}
\end{align}
\end{prop}

  We further apply Proposition \ref{prop: Fbound} in Section \ref{section-Patterson} to establish the following bound on $H(r,s;\psi)$. 
\begin{prop}
\label{lemma:laundrylist}
  With the notation as above and $H(r,s;\psi)$ defined in \eqref{h}, for any fixed $r$, $H(r,s;\psi)$ is analytical in the region when $3/2-1/8<\Re(s) \leq 3/2$.
Moreover, in this region,  we have for any $\varepsilon>0$,
\begin{align}
\label{hbound}
 & H(r,s;\psi) \ll N(r)^{(3/2-\Re (s))/2+\varepsilon}(1+|s|).
\end{align}
\end{prop}

  Applying Proposition \ref{lemma:laundrylist} which is reminiscent of \cite[Proposition 8.2]{DDHL}, together with the treatments in \cite{G&Zhao17} using multiple Dirichlet series, then allows us to complete the proofs of Theorems \ref{Theorem for all characters}--\ref{Theorem one-level densityQ} in Section \ref{sec: mainthm}.  We end this section by pointing out that the writing of this paper is largely influenced by the work of S. J. Patterson \cite{P}, T. Suzuki \cite{Suz1}, C. David and A. M. G\"{u}lo\u{g}lu \cites{DG22}, as well as C. David, A. Dunn, A. Hamieh, and H. Lin \cite{DDHL}. 

\section{Preliminaries}
\label{sec 2}

\subsection{Quartic residue symbols and Gauss sums}
\label{sec2.4}
   Recall that $K$ stands for the Gaussian field $\mq(i)$ throughout.  It is well-known that $K$ has class number one, $\mathcal{O}_{K}=\mz[i]$ and $U_K=\{ \pm 1, \pm i \}$.  As usual, $\mathcal O^*_K = \mathcal O_K\backslash \{0\}$.  Recall also that we set $\lambda=1+i$  and that it follows from 
\cite[\S 9.7, Lemma 7]{I&R} that every ideal in $\mathcal O_K$ co-prime to $2$ has a unique generator congruent to $1$ modulo $\lambda^3$.  This generator is called primary. It follows from \cite[\S 9.7, Lemma 6]{I&R} that an element $n=a+bi$ in $\mathcal O_K$ with $a, b \in \mz$ is primary if and only if $a \equiv 1 \pmod{4}$, and $b \equiv 0 \pmod{4}$ or $a \equiv 3 \pmod{4}$, and $b \equiv 2 \pmod{4}$. \newline
   
For any prime $\varpi$ co-prime to $2$ in $\mathcal{O}_{K}$, the quartic residue symbol $\leg {\cdot}{\cdot}_4$ is defined so that
$\leg{a}{\varpi}_4 \equiv
a^{(N(\varpi)-1)/4} \pmod{\varpi}$ with $\leg{a}{\varpi}_4 \in \{
\pm 1,  \pm i\}$ for any $a \in \mathcal{O}_{K}$ co-prime to $\varpi$ and that 
$\leg{a}{\varpi}_4 =0$ when $\varpi | a$. The above definition is then extended
to any composite $n \in \mathcal O_K$ co-prime to $2$ multiplicatively. We further define $\leg{a}{u}_4=1$ for any $a \in \mathcal O_K, u \in U_K$. We also write $\leg {\cdot}{\cdot}_2:=\leg {\cdot}{\cdot}^2_4$ for the quadratic symbol defined on $\mathcal O_K$. For $l=2$ and $4$, let $\chi_{l, a}$ be the symbol $\leg {\cdot}{a}_{l}$ for any $a \in \mathcal O_K$ co-prime to $2$ and $\chi^{(a)}_l$ the symbol $\leg {a}{\cdot}_{l}$ for any $a \in \mathcal{O}_K$.  As we reserve the letter $j$ for $4$, we shall also write $\chi_{j, a}, \chi^{(a)}_j$ for $\leg {\cdot}{a}_{4}$ and $\leg {a}{\cdot}_{4}$, respectively.\newline

  The quartic reciprocity law \cite[\S 9.9, Theorem~2]{I&R}  states that for any co-prime primary $\alpha, \gamma \in \mathcal O_K$, 
\begin{equation} \label{bilaw}
\Big( \frac{\alpha}{\gamma} \Big)_4=(-1)^{C(\alpha,\gamma)} \Big( \frac{\gamma}{\alpha} \Big)_4, \quad 
\end{equation}
where
\begin{equation} \label{Cdef}
C(\alpha,\gamma)=\frac{(N(\alpha)-1)}{4} \frac{(N(\gamma)-1)}{4}.
\end{equation}

    Moreover, we deduce from Lemma 8.2.1 and Theorem 8.2.4 in \cite{BEW} that the following supplementary laws hold for primary $n=a+bi$ with $a, b \in \mz$:
\begin{align}
\label{2.05}
  \leg {i}{n}_4=i^{(1-a)/2} \qquad \mbox{and} \qquad  \hspace{0.1in} \leg {1+i}{n}_4=i^{(a-b-1-b^2)/4}.
\end{align}

From \eqref{2.05}, we get that 
\begin{align*}
 \leg {i}{c}_{4}=\leg {1+i}{c}_{4}=1, \quad c \equiv 1 \pmod {16}.
\end{align*}

  It follows that $\chi_{j, c}$  can be regarded as a Hecke character of trivial infinite type for any $c \equiv 1 \pmod {16}$ and this character is primitive modulo $c$ if $c$ is a prime. Note that the symbol $\chi^{(a)}_{j}$  is also a Hecke character of trivial infinite type for any $a \in \mathcal O_K$. In the rest of the paper, we shall treat both $\chi_{j, \varpi}$ and $\chi^{(a)}_{j}$  as Hecke characters.  \newline
  
  Let $e(z) = \exp (2 \pi i z)$ for any $z \in \mc$ and set $\widetilde{e}(z) :=e(z+\overline{z})$.  For any $k, c \in \mathcal O_K$ with $(c, 2)=1$ and $l=2,4$, we define the associated Gauss sums $g_{K,l}(k, c)$ and $g_{l}(k, c)$ by
\begin{align*}
 g_{K,l}(k,c) = \sum_{x \shortmod{c}} \chi_{l,c}(x) \widetilde{e}\leg{kx}{c\sqrt{D_K}}, \quad g_{l}(k,c) = \sum_{x \shortmod{c}} \chi_{l,c}(x) \widetilde{e}\leg{kx}{c}, 
\end{align*}
  where $D_{K}$ is the discriminant of $K$ so that $D_{K}=-4$. Here we remark that we define two types of Gauss sums above as they both appear in the literature. We are in favor of using the notation $g_{K,l}(k,c)$ here since it emerges naturally in the functional equation of the related Hecke $L$-functions. However, we shall also make use of the $g_{l}(k,c)$ in some places of our treatments.  Note that we have
\begin{align}
\label{GKrel}
   g_{K,l}(k,c) =\leg {\sqrt{D_K}}{c}_l g_{l}(k,c).
\end{align}
  It follows from \eqref{2.05} that $\leg {\sqrt{D_K}}{\cdot}_l$ is a ray class group character modulo $16$ so that \eqref{GKrel} implies that $g_{K,l}(k,c)$ and  $g_{l}(k,c)$ only differ by a ray class group character modulo $16$. \newline
  
  Observe that for $(s,c)=1$, 
\begin{align}
\label{rel1}
   g_{K,l}(rs,c)=\overline{\Big( \frac{s}{c} \Big)_l} g_{K,l}(r,c), \qquad g_{l}(rs,c)=\overline{\Big( \frac{s}{c} \Big)_l} g_{l}(r,c). 
\end{align}

   It follows from \eqref{rel1} and \cite[Proposition 4.23]{Lemmermeyer}) that if $(r,c)=1$, 
\begin{align}
\label{gest}
   |g_{K,l}(r,c)|, \ |g_{l}(r,c)| \leq N(c)^{1/2}. 
\end{align}   

   Now, for $r$, $c$, $c^{\prime} \in \mathcal O_K$ with $(cc^{\prime},2)=1$ and $(c,c^{\prime})=1$, the Chinese remainder theorem implies that
\begin{equation}
\label{rel2}
g_{K,4}(r,cc^{\prime})= \Big( \frac{c}{c^{\prime}} \Big)_4 \Big( \frac{c^{\prime}}{c} \Big)_4 g_{K,4}(r,c) g_{K,4}(r,c^{\prime}) \qquad \mbox{and} \qquad g_4(r,cc^{\prime})= \Big( \frac{c}{c^{\prime}} \Big)_4 \Big( \frac{c^{\prime}}{c} \Big)_4 g_4(r,c) g_4(r,c^{\prime}).
\end{equation}
  Note that if $(r,c c^{\prime})=1$ the above relations continue to hold for $(c,c^{\prime}) \neq 1$ since by \eqref{rel1}, \eqref{rel2} and \eqref{rel4} below, both sides of the identities in \eqref{rel2} are $0$ in this case. \newline

  We deduce from \eqref{rel1} and \eqref{rel2} that for $c$, $c'$ primary and $(c,c^{\prime})=1$,
\begin{equation} 
\label{rel3}
g_{K,4}(r,cc^{\prime})=(-1)^{C(c^{\prime},c)} g_{K,4}(r (c^{\prime})^2,c) g_{K,4}(r,c^{\prime}) \qquad \mbox{and} \qquad g_4(r,cc^{\prime})=(-1)^{C(c^{\prime},c)} g_4(r (c^{\prime})^2,c) g_4(r,c^{\prime}).
\end{equation}

In the sequel, we shall write $g_{K, l}(c)$, $g_{l}(c)$ for $g_{K, l}(1, c)$, $g_{l}(1, c)$, respectively. We see from \eqref{GKrel} and \eqref{rel2} that it suffices to understand $g_{K,4}(\varpi^k,\varpi^{\ell})$ for primary primes $\varpi$ and non-negative rational integers $k$ and $\ell$.  To this end, we have from \cite[(2.9)]{Diac},
\begin{equation}
\label{rel4}
g_{K,4}(\varpi^k,\varpi^{\ell})=\begin{cases}
N(\varpi)^k g_{K,4}(\varpi) & \text{if } \ell=k+1, \quad k \equiv 0 \pmod{4},\\
N(\varpi)^k g_{K,4}(\varpi) & \text{if } \ell=k+1, \quad k \equiv 1 \pmod{4}, \\
N(\varpi)^k \big(\frac{-1}{\varpi} \big)_4 \overline{g_{K,4}(\varpi)} & \text{if } \ell=k+1, \quad k \equiv 2 \pmod{4}, \\
-N(\varpi)^k & \text{if } \ell=k+1, \quad k \equiv 3 \pmod{4}, \\
\varphi_{K}(\varpi^{\ell}) & \text{if } k \geq \ell, \quad \ell \equiv 0 \pmod{4}, \\
0 & \text{otherwise}.
\end{cases}
\end{equation}
Here $\varphi_{K}(n)$ denotes the number of elements in the reduced residue class of $\mathcal O_K/(n)$ for any $n \in \mathcal O_K$. 

\subsection{Functional equations for Hecke $L$-functions}
	
	For any primitive Hecke character $\chi$ of trivial infinite type modulo $q$, a well-known result of E. Hecke shows that $L(s, \chi)$ has an
analytic continuation to the whole complex plane and satisfies a
functional equation as given in \cite[Theorem 3.8]{iwakow}).  If $\chi=\chi_{j, \varpi}$ for a primary prime $\varpi \equiv 1 \pmod {16}$, this functional equations can be written as 
\begin{align}
\label{fneqnL}
  L(s, \chi_{j, \varpi})=g_{K,j}(\varpi)|D_K|^{1/2-s}N(\varpi)^{-s}(2\pi)^{2s-1}\frac {\Gamma(1-s)}{\Gamma (s)}L(1-s, \overline \chi_{j, \varpi}).
\end{align}

  We estimate the ratio $\Gamma(1-s)/\Gamma (s)$ by Stirling's formula (see \cite[(5.113)]{iwakow}), which yields for constants $c_0$, $d_0 \in \mr$,
\begin{align}
\label{Stirlingratio}
\begin{split}
  \Gamma(s) \ll e^{-|\Im(s)|} \qquad \mbox{and} \qquad \frac {\Gamma(c_0(1-s)+ d_0)}{\Gamma (c_0s+ d_0)} \ll (1+|s|)^{c_0(1-2\Re (s))}. 
\end{split}
\end{align}

\subsection{Estimations on $L$-functions}

The $L$-function $L(s, \chi)$ associate to any Hecke character $\chi$ of trivial infinite type has an Euler product for $\Re(s)$ large enough given by
\begin{align*}
 L(s, \chi)=\prod_{(\varpi)}\Big(1-\frac {\chi(\varpi)}{N(\varpi)^s}\Big)^{-1},
\end{align*}
  where $(n)$ denotes the ideal generated by any $n \in \mathcal O_K$ throughout the paper. Logarithmically differentiating both sides above renders that, for $\Re(s)$ large enough,
\begin{align*}
 -\frac {L'(s, \chi)}{L(s, \chi)}=\sum_{(n)}\frac {\Lambda_{K}(n)\chi(n)}{N(n)^s}.
\end{align*}

Now, let $m \in \mathcal O_K$ and any $\varpi \equiv 1 \pmod {16}$, we deduce, in a manner similar to \cite[(2.9), (2.12)-(2.14)]{G&Zhao17}, that under GRH, for any $\varepsilon>0$,
\begin{align}
\label{Lderboundgen}
\begin{split}
  (s-1)\cdot -\frac {L'(s, \chi^{(m)}_j)}{L(s, \chi^{(m)}_j)}  \ll & |s-1|\big((N(m)+2)(1+|s|)\big)^{\varepsilon}, \quad \Re(s) \geq 1/2+\varepsilon,  \\
  (s-1)L(s, \chi^{(m)}_j), \ (s-1)L(s, \chi_{j,\varpi}) \ll & |s-1|\big((N(m)+2)(1+|s|)\big)^{\varepsilon}, \quad \Re(s) \geq 1/2,  \\
  L(s,  \chi^{(m)}_j )^{-1} , \ L( s,  \chi_{j,\varpi} )^{-1} \ll & |sN(m)|^{\varepsilon}, \quad \Re(s) \geq 1/2+\varepsilon . \\
\end{split}
\end{align}

\subsection{The large sieve with quadratic symbols}
\label{quadsievesec}
 
  In \cite[Theorem 1]{DRHB}, D. R. Heath-Brown established a powerful and useful quadratic large sieve result over $\mathcal O_K$. Here we include an extension of such to the Gaussian field by K. Onodera \cite{Onodera}. See also \cite{GB} for results concerning general number fields.
\begin{theorem}
\label{quadsieve}
Let $\{a_n\}$ be an arbitrary complex sequence. Then we have
for $M,N \geq 1$, and any $\varepsilon>0$, 
\begin{align*}
\sum_{\substack{m \in \mathcal O_K \\ N(m) \leq M \\ m \equiv 1 \bmod{2} }} \mu^2_K(m) \Big | \hspace{0.1cm}
\sum_{\substack{n \in \mathcal O_K \\  N(n) \leq N \\ n \equiv 1 \bmod{2}}} a_n \mu^2_K(n) \Big( \frac{n}{m} \Big)_2 \Big |^2 \ll_{\varepsilon} (MN)^\varepsilon (M+N)
\sum_{N(n) \leq N} |a_{n}|^2 \mu^2_K(n).
\end{align*}
\end{theorem}

\subsection{Some results on multivariable complex functions} In order to apply the method of multiple Dirichlet series, we need some results from multivariable complex analysis. We begin by introducing the notation of a tube domain.
\begin{defin}
		An open set $T\subset\mc^n$ is a tube if there is an open set $U\subset\mr^n$ such that $T=\{z\in\mc^n:\ \Re(z)\in U\}.$
\end{defin}
	
   We define $T(U)=U+i\mr^n\subset \mc^n$ for any set $U\subset\mr^n$.  We note the following Bochner's Tube Theorem \cite{Boc}.
\begin{theorem}
\label{Bochner}
		Let $U\subset\mr^n$ be a connected open set and $f(z)$ a function holomorphic on $T(U)$. Then $f(z)$ has a holomorphic continuation to the convex hull of $T(U)$.
\end{theorem}

 Let the convex hull of any open set $T\subset\mc^n$ be denoted by $\widehat T$.  The following result \cite[Proposition C.5]{Cech1} bounds the modulus of holomorphic continuations of multivariable complex functions.
\begin{prop}
\label{Extending inequalities}
		Assume that $T\subset \mc^n$ is a tube domain, $g,h:T\rightarrow \mc$ are holomorphic functions, and let $\tilde g,\tilde h$ be their holomorphic continuations to $\widehat T$. If  $|g(z)|\leq |h(z)|$ for all $z\in T$ and $h(z)$ is nonzero in $T$, then also $|\tilde g(z)|\leq |\tilde h(z)|$ for all $z\in \widehat T$.
\end{prop}

\section{Metaplectic Eisenstein series}
\label{sec: MES}

  In our proof of Proposition \ref{lemg3}, we need to understand the analytical properties of certain Dirichlet series involving with quartic Gauss sums, twisted by Hecke characters. Such properties are essentially given in \cite[Lemma, p. 200]{P} by S. J. Patterson. However, we need a more precise estimation on the residues at the poles of these series. As it is shown by T. Kubota \cites{Kub1,Kub2} that Dirichlet series formed from quartic Gauss sums are intimately related to metaplectic Eisenstein series on the $4$-fold cover of $\operatorname{GL}_2$ over the Gaussian field and their theta functions, we review and develop here enough results for our purpose concerning these Eisenstein series. Our treatments follows largely those of T. Sukuki on the case without twists in \cite{Suz1}.

\subsection{The Kubota symbol} 
\label{multsec}

  We denote $\mathbb{H}^{3}$ the hyperbolic $3$-space $\mathbb{C} \times \mathbb{R}^{+}$ and embed $\mathbb{C}$ and $\mathbb{H}^3$ in the Hamilton quaternions  by identifying
$i=\sqrt{-1}$ with $\hat{i}$ and
$w=(z(w),v(w))=(z,v)=(x+iy,v) \in \mathbb{H}^3$ with $x+y \hat{i}+v \hat{j}$, where
$1,\hat{i},\hat{j},\hat{k}$ denote the unit quaternions.
We represent a point $w=(z,v)$ by the matrix $w=\pmatrix z {-v} {v} {\overline{z}}$,
and $u \in \mathbb{C}$ by the matrix $\widetilde{u}=\pmatrix u 0 0 {\overline{u}}$.
Then the group $\operatorname{SL}_2(\mathbb{C})$ acts discontinuously on $\mathbb{H}^3$ by
\begin{equation*}
\gamma w=\frac{\widetilde{a}w+\widetilde{b}}{\widetilde{c}w+\widetilde{d}}, \quad \gamma=\begin{pMatrix}
a b
c d
\end{pMatrix} \in \operatorname{SL}_2(\mathbb{C}), \quad \text{and} \quad
w \in \mathbb{H}^3.
\end{equation*}

 Let $N \in \mathcal O_K$ and $I_2$ the $2$ by $2$ identity matrix.  We define the congruence subgroup
\begin{equation*}
\Gamma_1(N):= \{ \gamma \in \operatorname{SL}_2(\mathcal O_K): \gamma \equiv I_2 \ppmod{N} \}.
\end{equation*}
 Then $\Gamma_1(N)$ also acts discontinuously on $\mathbb{H}^3$ whose fundamental domain has a finite
volume with respect to the $\operatorname{SL}_2(\mathbb{C})$-invariant
Haar measure $v^{-3} \dif z \dif v$ on $\mathbb{H}^3$, where we set $\dif z=\dif x \dif y$ for $z=x+iy$ with $x$, $y \in \mr$. The Kubota symbol
$\chi: \Gamma_1(N) \rightarrow \{\pm 1, \pm i\}$ is defined by
\begin{equation*}
\chi(\gamma)=\begin{cases}
\big( \frac{c}{a} \big)_4  & \text{if } c \neq 0, \\
1 & \text{if } c=0.
\end{cases}
\end{equation*}

From now on, we assume that $\lambda^4|N$ to note that it is shown via the quartic reciprocity law in \cite[p. 73]{Suz1} that $\chi$ is a character on $\Gamma_1(N)$ in this case. \newline

Let $\kappa \in P(\Gamma_1(N)) := \mathbb{Q}(i) \cup \{\infty = 1/0 \}$ be a cusp of $\Gamma_1(N)$ and denote $\Gamma_{\kappa}$ for the stabilizer of $\kappa$, so that
\begin{equation*}
\Gamma_{\kappa}=\{ \gamma \in \Gamma_1(N) : \gamma \kappa=\kappa \}.
\end{equation*}

 We write each $\kappa$ as $\alpha/\gamma$ with $(\alpha, \gamma)=1$. If $(\alpha, \lambda)=1$, then we may assume that $\alpha$ is primary. Similarly, if $(\alpha, \lambda) \neq 1$, then we may assume that
$\gamma$ is primary. It then follows (see \cite[Lemma 2]{Suz1}) that we have
\begin{lemma}
\label{lemcusp} Two cusps $\kappa=\alpha/\gamma$, $\kappa'=\alpha'/\gamma'$ are equivalent to each other under $\Gamma_1(N)$ if and
only $\alpha \equiv \alpha' \pmod N$ and $\gamma \equiv \gamma' \pmod N$.
\end{lemma}

For each cusp $\kappa=\alpha/\gamma$, we choose a scaling matrix $\sigma_{\kappa}=
\pmatrix \alpha \beta \gamma \delta \in \operatorname{SL}_2(\mathcal O_K)$ so that $\sigma_{\kappa}(\infty)=\alpha/\gamma$. It follows that $\Gamma_{\kappa}=\sigma_{\kappa} \Gamma_{\infty} \sigma_{\kappa}^{-1}$,
where
\begin{equation*}
\Gamma_{\infty}= \Big \{ \begin{pMatrix}
1 \mu
0 1
\end{pMatrix}: \mu \in N\mathcal O_K \Big \}.
\end{equation*}

  A cusp $\kappa \in P(\Gamma_1(N))$ is called essential if the restriction of $\chi$ to $\Gamma_{\kappa}$ is trivial. The above then implies that
a cusp $\kappa$  of $\Gamma(N)$ is essential if and only if for every $\mu \in N\mathcal O_K$,
\begin{equation*}
\chi \Big (\sigma_{\kappa} \begin{pMatrix}
1 \mu
0 1
\end{pMatrix}\sigma_{\kappa}^{-1} \Big )=1.
\end{equation*}

  It is shown in \cite[p. 73]{Suz1} that for $\kappa=\alpha/\gamma$,
\begin{equation*}
\chi \Big (\sigma_{\kappa} \begin{pMatrix}
1 \mu
0 1
\end{pMatrix}\sigma_{\kappa}^{-1} \Big )=\leg {\gamma}{1-\gamma \alpha \mu}_j\leg {\alpha}{1-\gamma \alpha \mu}^{-1}_j.
\end{equation*}

  We deduce from the quartic reciprocity law and the supplementary laws given in \eqref{bilaw} and \eqref{2.05} that the right-hand side expression above equals $1$ if $16|N$ so that the set of all essential cusps of $\Gamma (N)$ does not depend on $N$ in this case.

\subsection{Eisenstein series}
\label{metaeissec}

  Let $P^{\prime}(\Gamma_1(N))= \{ \kappa_1, \kappa_2, \kappa_3, \cdots \}$ be the complete set of representatives of inequivalent essential
cusps of $\Gamma_1(N)$. We fix $\kappa_1=\infty=1/0$ and write $\kappa_i=\alpha_i/\gamma_i$ such that we set $\gamma_i=1$ when $\kappa_i=0$. We also write for simplicity $\sigma_i=\sigma_{\kappa_i} $ and $\Gamma_i=\Gamma_{\kappa_i}$.  For each cusp $\kappa_i$, the Eisenstein series on $\Gamma_1(N)$ attached to the
essential cusp $\kappa_i \in P^{\prime}(\Gamma_1(N))$ is given by
\begin{equation}
\label{Eidef}
E_{i}(w,s,\Gamma_1(N)):=\sum_{\gamma \in \Gamma_{i} \backslash \Gamma_1(N) } \overline{\chi(\gamma)}
v(\sigma_{i}^{-1} \gamma w)^s, \quad w=(z(w), v(w)) \in \mathbb{H}^3, \quad \sigma := \Re(s) >2.
\end{equation}

 For $u,s \in \mathbb{C}$, we define 
\begin{align*}
\begin{split}
K(u,s):=& \int\limits_{\mathbb{C}} (|z|^2+1)^{-s} \widetilde{e}(-uz) \dif z = (2 \pi)^s |u|^{s-1} \Gamma(s)^{-1} K_{s-1}(4 \pi |u|),
\end{split}
\end{align*}
 where $K_{\xi}(\cdot)$ is the standard $K$-Bessel function of order $\xi \in \mathbb{C}$. Here we recall that $dz=dxdy$ for $z=x+iy$ with $x, y \in \mr$. We write $V(N)$ for the volume of $\mc/N\mathcal O_K$ with respect to $\dif z$. \newline

For two cusps $\kappa_i$ and $\kappa_l$, we write
\begin{align*}
M_{il}(N)=\Big \{ (c,d) \in  \mathcal O^*_K \times \mathcal O_K: \sigma_i \begin{pMatrix}
* *
c d
\end{pMatrix}\sigma_{l}^{-1} \in \Gamma_1(N)\Big \}.
\end{align*}
 We then define for $(c, d) \in M_{il}(N)$,
\begin{align*}
\overline \chi_{il}(c,d)=\overline \chi \Big (\sigma_i \begin{pMatrix}
* *
c d
\end{pMatrix}\sigma_{l}^{-1} \Big ).
\end{align*}

  We note that (see \cite[(3.6)]{Suz1}) the Eisenstein series $E_i$ defined in \eqref{Eidef} has the Fourier expansion at any essential cusp $\kappa_l \in P^{\prime}(\Gamma_1(N))$ given by
\begin{align}
\label{eisfourier}
E_{i}(\sigma_l w,s,\Gamma_1(N))=\delta_{il} v(w)^s+& \phi_{il}(s,0, \Gamma_1(N)) v(w)^{2-s} \nonumber \\
&+\sum_{\substack{\nu \in \mathcal O_K \\ \nu \neq 0 }} \phi_{il}(s,\nu,\Gamma_1(N)) v(w)^{2-s} K \Big( \frac{\nu v(w)}{N},s \Big) \widetilde{e} \Big ( \frac{\nu z(w)}{N} \Big),
\end{align}
where
\begin{align}
\begin{split}
\label{phidef}
 \delta_{il}= \begin{cases}
 1 \quad i=l, \\
 0 \quad i \neq l,
 \end{cases}, \qquad &
\phi_{il}(s,0,\Gamma_1(N)):= \pi V(N)^{-1}(s-1)^{-1} \sum_{\substack{(c,d) \in M_{il}(N)}}\overline \chi_{il}(c,d)N(c)^{-s}, \\
\phi_{il}(s, \nu ,\Gamma_1(N))& := V(N)^{-1} \sum_{\substack{(c,d) \in M_{il}(N)}}\overline \chi_{il}(c,d)\widetilde{e} \Big( \frac {d\nu}{cN} \Big) N(c)^{-s}, \quad \nu \neq 0.
\end{split}
\end{align}

  Now, we set $N=16m$ with $m$ being a primary square-free element of $\mathcal O_K$ and let $\kappa_i=\alpha/\gamma$ with $\alpha \equiv 0 \pmod{16m}$. Applying \eqref{eisfourier} to $\kappa_1=\infty$ with $\sigma_1$ being the identity matrix, we see that
\begin{align}
\label{EiFourier}
\begin{split}
  V(16m)\leg {\alpha}{\gamma}_jE_i(w, s, \Gamma_1(16m)) :=& \widetilde\phi_{i1}(s,0, m) v(w)^{2-s}+\sum_{\substack{\nu \in \mathcal O_K \\ \nu \neq 0 }} \widetilde\phi_{i1}(s,\nu,m) v(w)^{2-s} K \Big( \frac{\nu v(w)}{N},s \Big) \widetilde{e} \Big ( \frac{\nu z(w)}{N} \Big).
\end{split}
\end{align}

 As $\alpha \equiv 0 \pmod {16m}$, we argue in a way similar to the derivation of the expression for $\psi(s, \mu, \kappa)$ in \cite[p. 78]{Suz1} to get that for $\nu \neq 0$, 
\begin{align*}
\widetilde\phi_{i1}(s, \nu, m)=\sum_{c \equiv -\gamma \bmod{16m}}\sum_{\substack{d \bmod {16mc} \\ d \equiv 0 \bmod{16m}}}\leg {d}{c}_j\widetilde{e} \Big( \frac {d\nu}{16mc} \Big) N(c)^{-s}.
\end{align*}
For $\nu =0$, an additional factor of $\pi(s-1)^{-1}$ is required in the right-hand side of the above. \newline
  
 Upon writing $d=\delta 16m$ with $\delta$ running through a complete set of reduced residues modulo $c$, we infer that for $\nu \neq 0$,
\begin{align}
\label{phiexp}
\widetilde\phi_{i1}(s, \nu, m)=\sum_{c \equiv -\gamma \bmod{16m}}\leg {m}{c}_jg_j(\nu, c)N(c)^{-s}.
\end{align}
 A similar expression holds for $\widetilde\phi_{i1}(s, 0, m)$.  The above implies that $\phi_i$, and hence $V(16m)\leg {\alpha}{\gamma}_jE_i(w, s, \Gamma_1(16m))$, depends on $\gamma \pmod{16m} $ only and not on $\alpha$.  For any $\alpha \equiv 0 \pmod {16m}$, we shall henceforth use $\phi(s, \nu, \gamma \pmod{16m})$ to denote the quantity $\widetilde\phi_{i1}(s,\nu,m)$ defined in \eqref{EiFourier}.  Thus define for $\kappa_i=\alpha/\gamma$,
\begin{align}
\label{Egammadef}
 E(w, s, \gamma \shortmod{16m}) :=V(16m)\leg {\alpha}{\gamma}_jE_i(w, s, \Gamma_1(16m)).
\end{align}

Now, for any ray class character $\psi$  modulo $16$, define
\begin{align}
\label{E12def}
\begin{split}
 E(w, s, \psi, m) =& \sum_{\substack{ \gamma \bmod {16m}}}\psi(\gamma)E(w, s, \gamma \shortmod{16m}),  \quad \mbox{and} \\
 E_k(w, s, \psi, m) =& \sum_{\substack{ \gamma \bmod {16m} \\ \gamma \equiv k \bmod {16}}}\psi(\gamma)E(w, s, \gamma \shortmod{16m}), \quad   \text{$k$ primary}.
\end{split}
\end{align}
Further define, for any $\nu \in \mathcal O_K$,
\begin{align*}
 \phi(s, \nu, \psi, m;k) =& \sum_{\substack{ \gamma \bmod {16m} \\ \gamma \equiv k \bmod {16}}}\psi(\gamma)V(16m)\leg {\alpha}{\gamma}_j\phi(s, \nu, \gamma \shortmod{16m}).
\end{align*}

  We deduce from \eqref{phiexp} that for any $\nu \neq 0$, 
\begin{align}
\label{E12exp}
 \phi(s, \nu, \psi, m;k)  =& \sum_{\substack{c \equiv -k \bmod{16} \\ (c, m)=1}}\psi(c)\leg {m}{c}_jg_j(\nu, c)N(c)^{-s}.
\end{align}
  We point out here that it follows from \eqref{EiFourier} and \eqref{Egammadef} that the functions $\phi(s, \nu, \psi, m;k) $ are the coefficients of the Fourier expansion of $E(w, s, \psi, m)$ at the infinity.

\subsection{Singular values of Eisenstein series}
\label{singularvalues}

  Let $E(w, s, \gamma \pmod{16m})$ be given as in \eqref{Egammadef}. We define its singular value $\phi(s, \gamma \pmod {16m}, \kappa_l)$ at $\kappa_l=\alpha'/\gamma' \in P'(\Gamma_1(16m))$ to be $V(16m)\leg {\alpha}{\gamma}_j\phi_{il}(s,0,\Gamma_1(16m))$, where $\phi_{il}$ is defined in \eqref{phidef} and we set $\kappa_i=\alpha/\gamma$ with $\alpha \equiv 0 \pmod {16m}$. Then we deduce from \eqref{phidef} that
\begin{align}
\label{psidef}
\phi(s, \gamma \shortmod {16m}, \kappa_l)= \pi(s-1)^{-1}\leg{\alpha}{\gamma}_j \sum_{\substack{(c,d) \in M_{il}(N)}}\overline \chi_{il}(c,d)N(c)^{-s}.
\end{align}

  For any ray class character $\psi$ modulo $16$, let $E_k(w, s, \psi, m)$ be given as in \eqref{E12def} for any primary $k$. We define the singular value of $E_k(w, s, \psi, m)$ at $\kappa_l$ to be
\begin{align*}
\phi(s, \psi, m, \kappa_l;k)= \sum_{\substack{ \gamma \bmod {16m} \\ \gamma \equiv k \bmod {16}}}\psi(\gamma)\phi(s, \gamma \shortmod{16m}, \kappa_l).
\end{align*}
  We further define
\begin{align}
\label{psiexpression}
\phi(s, \psi, m, \kappa_l)= \sum_{\substack{ \gamma \bmod {16m} }}\psi(\gamma)\phi(s, \gamma \shortmod {16m}, \kappa_l).
\end{align}
  If we write
\begin{align*}
\sigma_i= \begin{pMatrix}
\alpha \beta
\gamma \delta
\end{pMatrix}, \quad
\sigma_l= \begin{pMatrix}
{\alpha'} {\beta'}
{\gamma'} {\delta'}
\end{pMatrix},
\end{align*}
 then similar to the expression given in the last display on \cite[p. 106]{Suz1}, we have
\begin{align*}
M_{il}(16m)=\Big \{ (c,d) \in  \mathcal O^*_K \times \mathcal O_K: c \equiv -\gamma \alpha' \pmod {16m}, \quad d \equiv -\gamma \beta' \pmod {16m} \Big \}.
\end{align*}

  When $\lambda \nmid \gamma'$ (resp. $\lambda \nmid \alpha'$), we may take the value of $\delta'$ (resp. $\beta'$) such that it is divisible by a high power of $\lambda$. Our next lemma evaluates $\overline \chi_{il}(c,d)$.
\begin{lemma}
\label{lemchieval} With the notation as above, we have
\begin{align*}
 \overline \chi_{il}(c,d) =
\begin{cases}
  \leg {\alpha}{\gamma}^{-1}_j\leg {-\alpha'}{\gamma'}_j\leg {c}{d}_j(-1)^{C(\gamma', \gamma)}, \quad \lambda \nmid \gamma', \\
  \leg {\alpha}{\gamma}^{-1}_j\leg {\gamma'}{\alpha'}_j\leg {d}{c}_j(-1)^{C(\alpha', \gamma)},  \quad \lambda \nmid \alpha',
\end{cases}
\end{align*}
  where $C(\cdot,\cdot)$ is given in \eqref{Cdef}. 
\end{lemma}
\begin{proof}
We take the convention that $\gamma'$ (resp. $\alpha'$) is primary when  $\lambda \nmid \gamma'$ (resp. $\lambda \nmid \alpha'$). It then follows from the conditions $d \equiv -\gamma \beta' \pmod {16m}, \alpha'\delta'-\gamma'\beta' =1$ that if $\lambda \nmid \gamma'$, then $d\gamma'-c\delta' \equiv d\gamma' \equiv -\gamma \gamma'\beta' \equiv 1 \pmod {\lambda^3}$. Similarly, using the conditions $c \equiv -\gamma \alpha' \pmod {16m}, d \equiv -\gamma \beta' \pmod {16m}, \alpha'\delta'-\gamma'\beta' =1$ we have $d\gamma'-c\delta' \equiv  -c\delta' \equiv \gamma \alpha' \delta' \equiv 1 \pmod {\lambda^3}$ when $\lambda \nmid \alpha'$. Thus, in either case, $d\gamma'-c\delta'$ is primary.  Then applying \cite[Lemma 3]{Suz1} yields
\begin{align*}
\begin{split}
\overline \chi_{il}(c,d)=& \overline \chi \Big (
\begin{pMatrix}
\alpha \beta
\gamma \delta
\end{pMatrix}
\begin{pMatrix}
a b
c d
\end{pMatrix}
\begin{pMatrix}
{\alpha'} {\beta'}
{\gamma'} {\delta'}
\end{pMatrix}^{-1} \Big )
= \overline \chi \Big (
\begin{pMatrix}
\alpha \beta
\gamma \delta
\end{pMatrix}
\begin{pMatrix}
{a\delta'-b\gamma'} {-a\beta'+b\alpha'}
{c\delta'-d\gamma'} {-c\beta'+d\alpha'}
\end{pMatrix}
\Big )
= \leg {\alpha}{\gamma}^{-1}_j\leg {-c\beta'+d\alpha'}{d\gamma'-c\delta'}_j.
\end{split}
\end{align*}

  When $\lambda \nmid \gamma'$, we write for simplicity $e=(c, \gamma')$ with $e$ being primary. It follows that $(e, 2)=1$ and $(c/e, \gamma'/e)=1$. As $(\gamma', \delta')=(c,d)=1$, we see that $(c\gamma'/e^2, d\gamma'/e-c\delta'/e)=1$,  so that
\begin{align*}
\begin{split}
\leg {-c\beta'+d\alpha'}{d\gamma'-c\delta'}_j =& \leg {\gamma'/e}{d\gamma'/e-c\delta'/e}^{-1}_j\leg {\gamma'/e}{d\gamma'/e-c\delta'/e}_j\leg {-c\beta'+d\alpha'}{d\gamma'/e-c\delta'/e}_j\leg {-c\beta'+d\alpha'}{e}_j \\
=& \leg {\gamma'/e}{d\gamma'/e-c\delta'/e}^{-1}_j\leg {\gamma'/e}{d\gamma'/e-c\delta'/e}_j\leg {-c\beta'+d\alpha'}{d\gamma'/e-c\delta'/e}_j\leg {d\alpha'}{e}_j.
\end{split}
\end{align*}

Mark that
\begin{align*}
\gamma'(-c\beta'+d\alpha')=-c\beta'\gamma'+d\gamma'\alpha'=c(1-\alpha'\delta')+d\gamma'\alpha'=c+\alpha'(d\gamma'-c\delta').
\end{align*}

  It follows that
\begin{align*}
\begin{split}
\leg {-c\beta'+d\alpha'}{d\gamma'-c\delta'}_j
=& \leg {\gamma'/e}{d\gamma'/e-c\delta'/e}^{-1}_j\leg {c/e}{d\gamma'/e-c\delta'/e}_j\leg {d\alpha'}{e}_j.
\end{split}
\end{align*}

  As $\delta'$ is divisible by a high power of $\lambda$ and $(e, 2)=1$, $c\delta'/e$ is also divisible by a high power of $\lambda$ so that
\begin{align*}
\leg {c/e}{d\gamma'/e-c\delta'/e}_j=\leg {c/e}{d\gamma'/e}_j.
\end{align*}

 Observe further that $d\gamma'-c\delta' \equiv \gamma \pmod {16m}$ so that
\begin{align*}
 \frac {N(d\gamma'/e-c\delta'/e)-1}4+\frac {N(e)-1}4 \equiv \frac {N(a(d\gamma'/e-c\delta'/e))-1}4 \equiv\frac {N(d\gamma'-c\delta')-1}4 \equiv \frac {N(\gamma)-1}4  \pmod 2.
\end{align*}
  Note also that
\begin{align*}
 \frac {N(\gamma'/e)-1}4+\frac {N(e)-1}4 \equiv \frac {N(e \cdot \gamma'/e)-1}4 =\frac {N(\gamma')-1}4   \pmod 2.
\end{align*}
  We then deduce that
\begin{align*}
 \leg {\gamma'/e}{d\gamma'/e-c\delta'/e}^{-1}_j=\leg {-c\delta'/e}{\gamma'/e}^{-1}_j(-1)^{C(\gamma'/e, d\gamma'/e-c\delta'/e)}=\leg {-c\delta'/e}{\gamma'/e}^{-1}_j(-1)^{(\frac {N(\gamma')-1}4-\frac {N(e)-1}4)(\frac {N(\gamma)-1}4-\frac {N(e)-1}4)}.
\end{align*}

The conditions $d \equiv -\gamma \beta' \pmod {16m}, \alpha'\delta'-\gamma'\beta' =1$ imply that $d$ primary.  The above leads to
\begin{align*}
\overline \chi_{ij}(c,d)=& \leg {\alpha}{\gamma}^{-1}_j\leg {-\delta'}{\gamma'/e}^{-1}_j\leg {c/e}{d}_j\leg {d\alpha'}{e}_j (-1)^{(\frac {N(\gamma')-1}4-\frac {N(e)-1}4)(\frac {N(\gamma)-1}4-\frac {N(e)-1}4)} \\
=& \leg {\alpha}{\gamma}^{-1}_j\leg {\alpha'}{\gamma'}_j\leg {-1}{\gamma'/e}^{-1}_j\leg {c/a}{d}_j\leg {d}{e}_j(-1)^{(\frac {N(\gamma')-1}4-\frac {N(e)-1}4)(\frac {N(\gamma)-1}4-\frac {N(e)-1}4)}\\
=& \leg {\alpha}{\gamma}^{-1}_j\leg {\alpha'}{\gamma'}_j\leg {-1}{\gamma'}_j\leg {c}{d}_j\leg {-1}{e}_j(-1)^{C(d,e)}(-1)^{(\frac {N(\gamma')-1}4-\frac {N(e)-1}4)(\frac {N(\gamma)-1}4-\frac {N(e)-1}4)},
\end{align*}
 where the second equality above follows by noting that $\alpha'\delta' \equiv 1 \pmod {\gamma'}$. \newline

  From $\alpha'\delta'-\beta'\gamma'=1$, we see that if $\delta'$ is divisible by a high power of $\lambda$, then $\beta'\gamma' \equiv -1 \pmod {16}$ so that
\begin{align*}
 \frac {N(\gamma')-1}4+\frac {N(\beta')-1}4 \equiv \frac {N(\beta' \cdot \gamma')-1}4 \equiv \frac {N(-1)-1}4 \equiv 0  \pmod 2.
\end{align*}
 The above implies that
\begin{align*}
 \frac {N(\beta')-1}4 \equiv \frac {N(\gamma')-1}4   \pmod 2.
\end{align*}
  Thus, as $d \equiv -\gamma \beta' \pmod {16m}$,
\begin{align*}
 \frac {N(d)-1}4 \equiv \frac {N( -\gamma \beta')-1}4 \equiv \frac {N(\gamma)-1}4+ \frac {N(\beta')-1}4 \equiv \frac {N(\gamma)-1}4+ \frac {N(\gamma')-1}4 \pmod 2.
\end{align*}

  We apply the above to see that
\begin{align*}
 \leg {-1}{e}_j(-1)^{C(d,e)} & (-1)^{(\frac {N(\gamma')-1}4-\frac {N(e)-1}4)(\frac {N(\gamma)-1}4-\frac {N(e)-1}4)} \\
=& (-1)^{(\frac {N(e)-1}{4}+(\frac {N(\gamma)-1}4+ \frac {N(\gamma')-1}4)\cdot \frac {N(e)-1)}{4}+(\frac {N(\gamma')-1}4-\frac {N(e)-1}4)(\frac {N(\gamma)-1}4-\frac {N(e)-1}4))} =  (-1)^{C(\gamma', \gamma)}.
\end{align*}

  We thus conclude that when  $\lambda \nmid \gamma'$, we have
\begin{align}
\label{Mijcomp0}
\overline \chi_{il}(c,d)=&  \leg {\alpha}{\gamma}^{-1}_j\leg {-\alpha'}{\gamma'}_j\leg {c}{d}_j(-1)^{C(\gamma', \gamma)}.
\end{align}

   When $\lambda \nmid \gamma'$, we write for simplicity $f=(d, \delta')$ with $f$ being primary. It follows that $(f, 2)=1$ and $(d/f, \delta'/f)=1$. As we have $(\gamma', \delta')=(c,d)=1$, we see that $(d\delta'/f^2, d\gamma'/f-c\delta'/f)=1$,  so that we have
\begin{align*}
\begin{split}
\leg {-c\beta'+d\alpha'}{d\gamma'-c\delta'}_j =& \leg {\delta'/f}{d\gamma'/f-c\delta'/f}^{-1}_j\leg {\delta'/f}{d\gamma'/f-c\delta'/f}_j\leg {-c\beta'+d\alpha'}{d\gamma'/f-c\delta'/f}_j\leg {-c\beta'+d\alpha'}{f}_j \\
=& \leg {\delta'/f}{d\gamma'/f-c\delta'/f}^{-1}_j\leg {\delta'/f}{d\gamma'/f-c\delta'/f}_j\leg {-c\beta'+d\alpha'}{d\gamma'/f-c\delta'/f}_j\leg {-c\beta'}{f}_j.
\end{split}
\end{align*}

  Note that
\begin{align*}
\delta'(-c\beta'+d\alpha')=-c\beta'\delta'+d\delta'\alpha'=d(1+\gamma'\beta')-c\beta'\delta'=d+\beta'(d\gamma'-c\delta').
\end{align*}

  It follows that
\begin{align*}
\begin{split}
\leg {-c\beta'+d\alpha'}{d\gamma'-c\delta'}_j
=& \leg {\delta'/f}{d\gamma'/f-c\delta'/f}^{-1}_j\leg {d/f}{d\gamma'/a-c\delta'/a}_j\leg {-c\beta'}{f}_j.
\end{split}
\end{align*}

  Recall that when $\lambda \nmid \alpha'$, the value of $\beta'$ is chosen so that it is divisible by a high power of $\lambda$. For this $\beta'$, the condition $d \equiv -\gamma \beta' \pmod {16}$ implies that $d \equiv 0 \pmod {16}$. As $(f, 2)=1$, this implies that $d\gamma'/f \equiv 0 \pmod {16}$ so that we have
\begin{align*}
\leg {d/f}{d\gamma'/f-c\delta'/f}_j=\leg {d/f}{c\delta'/f}_j.
\end{align*}

  Again using  $d\gamma'-c\delta' \equiv \gamma \pmod {16m}$, we see that
\begin{align*}
 \frac {N(d\gamma'/f-c\delta'/f)-1}4+\frac {N(f)-1}4 \equiv \frac {N(f(d\gamma'/f-c\delta'/f))-1}4 \equiv\frac {N(d\gamma'-c\delta')-1}4 \equiv \frac {N(\gamma)-1}4  \pmod 2.
\end{align*}
Furthermore,
\begin{align*}
 \frac {N(\delta'/f)-1}4+\frac {N(f)-1}4 \equiv \frac {N(f \cdot \delta'/f)-1}4 =\frac {N(\delta')-1}4   \pmod 2.
\end{align*}
Thus it can be deduced that
\begin{align*}
 \leg {\delta'/f}{d\gamma'/f-c\delta'/f}^{-1}_j=\leg {d\gamma'/f}{\delta'/f}^{-1}_j(-1)^{C(\delta'/f, d\gamma'/f-c\delta'/f)}=\leg {d\gamma'/f}{\delta'/f}^{-1}_j(-1)^{(\frac {N(\delta')-1}4-\frac {N(f)-1}4)(\frac {N(\gamma)-1}4-\frac {N(f)-1}4)}.
\end{align*}

 Observe that the condition $c \equiv -\gamma \alpha' \pmod {16m}$ implies that $-c$ primary. We apply the above to see that
\begin{align*}
\overline \chi_{ij}(c,d)=& \leg {\alpha}{\gamma}^{-1}_j\leg {\gamma'}{\delta'/f}^{-1}_j\leg {d/f}{c}_j\leg {-c\beta'}{f}_j (-1)^{(\frac {N(\delta')-1}4-\frac {N(f)-1}4)(\frac {N(\gamma)-1}4-\frac {N(f)-1}4)} \\
=& \leg {\alpha}{\gamma}^{-1}_j\leg {\gamma'}{\delta'}^{-1}_j\leg {d/f}{c}_j\leg {c}{f}_j(-1)^{(\frac {N(\delta')-1}4-\frac {N(f)-1}4)(\frac {N(\gamma)-1}4-\frac {N(f)-1}4)}\\
=& \leg {\alpha}{\gamma}^{-1}_j\leg {\gamma'}{\delta'}^{-1}_j\leg {d}{c}_j\leg {-1}{f}_j(-1)^{C(-c,f)}(-1)^{(\frac {N(\delta')-1}4-\frac {N(f)-1}4)(\frac {N(\gamma)-1}4-\frac {N(f)-1}4)},
\end{align*}
 where the second equality above follows by noting that $-\beta'\gamma' \equiv 1 \pmod {\delta'}$. \newline

  As $\alpha'\delta'-\beta'\gamma'=1$, if $\beta'$ is divisible by a high power of $\lambda$, then $\alpha'\delta' \equiv 1 \pmod {16}$ so that
\begin{align*}
 \frac {N(\alpha')-1}4+\frac {N(\delta')-1}4 \equiv \frac {N(\alpha' \cdot \delta')-1}4 \equiv  0  \pmod 2.
\end{align*}
 This implies that
\begin{align*}
 \frac {N(\delta')-1}4 \equiv \frac {N(\alpha')-1}4   \pmod 2.
\end{align*}
  Also, as $c \equiv -\gamma \alpha' \pmod {16m}$, we see that
\begin{align*}
 \frac {N(-c)-1}4 \equiv \frac {N( \gamma \alpha')-1}4 \equiv \frac {N(\gamma)-1}4+ \frac {N(\alpha')-1}4  \pmod 2.
\end{align*}

Applying the above yields that
\begin{align*}
 \leg {-1}{f}_j(-1)^{C(-c,f)}& (-1)^{(\frac {N(\delta')-1}4-\frac {N(f)-1}4)(\frac {N(\gamma)-1}4-\frac {N(f)-1}4)} \\
& =(-1)^{\frac {N(f)-1}{4}+(\frac {N(\gamma)-1}4+ \frac {N(\alpha')-1}4)\cdot \frac {N(f)-1}{4}+(\frac {N(\alpha')-1}4-\frac {N(f)-1}4)(\frac {N(\gamma)-1}4-\frac {N(f)-1}4)} =  (-1)^{C(\alpha', \gamma)}.
\end{align*}

  Moreover, we deduce from $\alpha'\delta'-\beta'\gamma'=1$ and that the quantity $\beta'$ is divisible by a high power of $\lambda$ that
\begin{align*}
 & 1=\leg {\gamma'}{\alpha'\delta'-\beta'\gamma'}_j=\leg {\gamma'}{\alpha'\delta'}_j.
\end{align*}
  It follows that
\begin{align*}
 & \leg {\gamma'}{\delta'}^{-1}_j=\leg {\gamma'}{\alpha'}_j.
\end{align*}

  We thus conclude that when  $\lambda \nmid \alpha'$, we have
\begin{align}
\label{Mijcomp1}
\begin{split}
\overline \chi_{il}(c,d)
=& \leg {\alpha}{\gamma}^{-1}_j\leg {\gamma'}{\alpha'}_j\leg {d}{c}_j(-1)^{C(\alpha', \gamma)}.
\end{split}
\end{align}

  The lemma now follows from \eqref{Mijcomp0} and \eqref{Mijcomp1}.
\end{proof}

\subsection{Relations between residues}
\label{Sec: resrel}

As given on \cite[p. 76]{Suz1} and \cite[p. 648]{Diac}, the Eisenstein series defined in \eqref{E12def} are holomorphic in the region $\Re (s)>1$ except for a possible simple pole at $s =5/4$.  With the aid of Lemma \ref{lemchieval}, our next result establishes an analogue to \cite[Proposition 9]{Suz1}, which allows us to compare the corresponding residues. 
\begin{prop}
\label{lemgammaeven} With the notation as above, let $\psi$ be any ray class character modulo $16$. Then for any primary $k$, we have
\begin{align}
\label{Resrel}
 \res_{s=5/4} E_k(w, s, \psi, 1)=\prod_{\substack{\varpi \odd \\ \varpi|m}}(1+N(\varpi)^{-1})\res_{s=5/4} E_k(w, s, \psi, m),
\end{align}
  where $E_k(w, s, \psi, m)$ is defined in \eqref{E12def}.
\end{prop}
\begin{proof}
Firstly,
\begin{align*}
 E_k(w, s, \psi, m) =& \frac {1}{\#h_{(16)}}\sum_{\mathfrak{b}}\mathfrak{b}((-k)^{-1})\sum_{\substack{ \gamma \bmod {16m} }} \psi(\gamma)\mathfrak{b}(\gamma)E(w, s, \gamma \shortmod{16m}),
\end{align*}
 where $\mathfrak{b}$ runs over all ray class characters modulo $16$. It therefore suffices to establish \eqref{Resrel} with $E_k$ replaced by $E$ throughout. Further, it follows from \cite[Lemma, p. 200]{P} that in this case, the residue on both sides of \eqref{Resrel} is $0$ unless $\psi^4$ is the principal Hecke character, a condition we shall henceforth assume throughout the proof. \newline

  We now apply the principle stated on \cite[p. 77]{Suz1} to see that it suffices to show that the singular values of both sides of \eqref{Resrel} (with $E_k$  replaced by $E$ throughout) match with each other. Thus, it remains to show that
\begin{align}
\label{Resphirel}
 \res_{s=5/4} \phi(s, \psi, 1, \kappa_l)=\prod_{\substack{\varpi \odd \\ \varpi|m}}(1+N(\varpi)^{-1}) \res_{s=5/4} \phi(s, \psi, m, \kappa_l),
\end{align}
where $\phi$ is given in \eqref{psiexpression}. \newline

 To do so, we fix an $\eta$ co-prime to $\lambda$ and we write $c=i^a\lambda^bMc'$ with $0 \leq a \leq 3$, $b \geq 0$, $M|m^{\infty}$, with $M$, $c'$ primary and
$(c', m)=1$. Then there exists $A_i, 1 \leq i \leq 3$ such that
\begin{align*}
 mMc'A_1+\lambda^{b+8}c'A_2+\lambda^{b+8}mMA_3=1.
\end{align*}
  Let
\begin{align*}
 d=mMc'A_1d_1+\lambda^{b+8}c'A_2d_2+\lambda^{b+8}mMA_3d_3,
\end{align*}
  where $d_1$ runs through a set of residues modulo $\lambda^{b+8}$ subject to $d_1 \equiv \eta \pmod{16}$, $d_2$ runs through
a set of residues  modulo $mM$ subject to $d_2 \equiv \eta \pmod m$, and $d_3$  runs through a complete set
of reduced residues modulo $c'$. Then
\begin{align*}
 \leg {c}{d}_j=\leg {i^a\lambda^b}{d}_j\leg {Mc'}{d}_j=(-1)^{C(d, Mc')}\leg {i^a\lambda^b}{d}_j\leg {d_2}{M}_j\leg {d_3}{c'}_j=(-1)^{C(\eta, Mc')}\leg {i^a\lambda^b}{d_1}_j\leg {d_2}{M}_j\leg {d_3}{c'}_j.
\end{align*}
  It follows that
\begin{align*}
\begin{split}
\sum_{\substack{d \pmod {16mc} \\ d \equiv \eta \pmod{16m}}}\leg {c}{d}_j= &(-1)^{C(\eta, Mc')}\sum_{\substack{d_1 \pmod {\lambda^{b+8}} \\ d_1 \equiv \eta \pmod{16}}}\leg {i^a\lambda^b}{d_1}_j\sum_{\substack{d_2 \pmod {mM} \\ d_2 \equiv \eta \pmod{m}}}\leg {d_2}{M}_j\sum_{\substack{d_3 \pmod {c'}}}\leg {d_3}{c'}_j \\
=&
\begin{cases}
 (-1)^{C(\eta, Mc')} \leg {i^a\lambda^b}{\eta}_j2^b\leg {\eta}{M}_jN(M)\varphi_K(c'), & \text{when $c'$ is a fourth power}, \\
 0, & \text{otherwise}.
\end{cases}
\end{split}
\end{align*}

   We thus deduce from \eqref{psidef}, \eqref{psiexpression}, Lemma \ref{lemchieval} and the above that if $\lambda \nmid \gamma'$,
\begin{align}
\label{sumpsi}
\begin{split}
   \phi(s, \psi, m, \kappa_l) =& \frac{\pi}{s-1} \leg {-\alpha'}{\gamma'}_j  \sum_{\substack{\gamma \bmod {16m} }} (-1)^{C(\gamma', \gamma)} \psi(\gamma) \sum_{c \equiv -\gamma \alpha' \bmod{16m}}\sum_{\substack{d \bmod {16mc} \\ d \equiv -\gamma \beta' \bmod{16m}}}\leg {c}{d}_jN(c)^{-s} \\
=&  \frac{\pi}{s-1} \leg {-\alpha'}{\gamma'}_j  \sum_{0 \leq a \leq 3}\sum_{b \geq 1}2^{b(1-s)}\sum_{\substack{M |m^{\infty} \\ M \odd}}N(M)^{1-s} \sum_{\substack{(c, m)=1 \\ c \odd}}\frac{\varphi_K(c^4)}{N(c)^{4s}} \\
& \hspace*{1cm} \times \sum_{\substack{\gamma \bmod {16m} \\ -\gamma \alpha' \equiv i^a\lambda^{b}Mc^4 \bmod {16m}} } (-1)^{C(\gamma', \gamma)} (-1)^{C(-\gamma\beta', M)} \psi(\gamma)  \leg {i^a\lambda^b}{\gamma \beta'}_j\leg {-\gamma \beta'}{M}_j,
\end{split}
\end{align}
  where the last equality above follows by noting that for any primary $c$,
\begin{align*}
  C(-\gamma\beta', Mc^4) \equiv  C(-\gamma\beta', M)  \pmod 2.
\end{align*}

  Suppose that $m=m_1m_2, m_1|\alpha', (\alpha', m_2)=1$. Then the relation $-\gamma \alpha' \equiv i^a\lambda^{b}Mc^4 \pmod {16m}$ implies that $m_1|M$. Moreover, one must have $(M, m_2)=1$ since otherwise the condition $-\gamma \alpha' \equiv i^a\lambda^{b}Mc^4 \pmod {16m}$ implies that $\gamma \equiv 0 \pmod {m_2}$ so that $\leg {\gamma}{M}_j=0$. Thus the condition $-\gamma \alpha' \equiv i^a\lambda^{b}Mc^4 \pmod {16m}$ becomes $-\gamma \alpha' \equiv i^a\lambda^{b}Mc^4 \pmod {16m_2}$. Upon writing $\gamma=16m_2\gamma'+m_1c_1$ with $m_1c_1$ satisfying the condition $-\alpha' m_1c_1 \equiv i^a\lambda^{b}Mc^4 \pmod {16m_2}$ and $\gamma'$ running over a complete residue class modulo $m_1$ and noticing that $\psi$ is a ray class character modulo $16$, it follows that the sum over such $\gamma \pmod {16m}$ contains a sum
  \[ \sum_{\gamma' \bmod {m_1}}\leg {\gamma'}{M}_j=0 \]
 unless $M$ is a fourth power. We may thus write $M=M'^4$ with $m_1|M'^4$ and note that the congruence condition satisfied by $\gamma$ now becomes $-\gamma \alpha' \equiv i^a\lambda^{b}M'^4c^4 \pmod {16m_2}$, which can be written equivalently as $-\gamma \alpha' \equiv i^a\lambda^{b}M'^4c^4 \pmod {16}, -\gamma \alpha' \equiv i^a\lambda^{b}M'^4c^4 \pmod {m_2}$. The last condition may be further written as $\gamma  \equiv -i^a\lambda^{b}M'^4c^4\alpha'^{-1} \pmod {m_2}$, where we write $\alpha'^{-1}$ for the inverse of $\alpha'$ modulo $m_2$. We further write $\alpha'=\lambda^u\eta$ with $(\eta, \lambda)=1$. When $u \geq 8$, then the condition $-\gamma \alpha' \equiv i^a\lambda^{b}M'^4c^4 \pmod {16}$ is superfluous provided that $b \geq 8$. When $0 \leq u \leq 7$, then we must have $b = u$ for otherwise $\psi(\gamma)=0$.
The condition $-\gamma \alpha' \equiv i^a\lambda^{b}M'^4c^4 \pmod {16}$ thus becomes $\gamma \equiv -i^aM'^4c^4\eta^{-1} \pmod {\lambda^{8-u}}$, where we write $\eta^{-1}$ for the inverse of $\eta$ modulo $\lambda^{8-u}$. Without loss of generality, we may assume $0 \leq u \leq 7$ so that we may write $\gamma=(-i^aM'^4c^4\eta^{-1}+\gamma_1M'^4c^4\lambda^{8-u})mC_1-i^a\lambda^{d}M'^4c^4\alpha'^{-1}16m_1C_2+\gamma_216m_2C_3$ with $C_1, C_2, C_3$ satisfying the conditions $mC_1 \equiv 1 \pmod {16}, 16m_1C_2 \equiv 1 \pmod {m_2}, 16m_2C_3 \equiv 1 \pmod {m_1}$ and $\gamma_1, \gamma_2$ running over the complete residue class modulo $\lambda^{t}$ and $m_1$, respectively. Thus, when $\psi^4$ is the principal Hecke character, we have $\psi(\gamma)=\psi(-i^a\eta^{-1}+\gamma_1\lambda^{8-u})$. Similarly, we have
\begin{align*}
\begin{split}
 \leg {i^a\lambda^b}{\gamma}_j= \leg {i^a\lambda^u}{-i^a\eta^{-1}+\gamma_1\lambda^{8-u}}_j \quad \mbox{and} \quad
(-1)^{C(\gamma', \gamma)}= (-1)^{C(\gamma', -i^aM'^4c^4\eta^{-1}+\gamma_1M'^4c^4\lambda^{8-u})}=(-1)^{C(\gamma', -i^a\eta^{-1}+\gamma_1\lambda^{8-u})},
\end{split}
\end{align*}
 where the last equality holds by noting that $N(M')^4 \equiv 1 \pmod 8$ when $(M', 2)=1$. \newline

 Moreover, as $\alpha'\delta'-\beta'\gamma'=1$, we see that $(\beta', \alpha')=1$, hence $(\beta', M)=1$. Thus we have $\leg {-\beta'}{M}_j=1$ when $M$ is a fourth power. The relation $\alpha'\delta'-\beta'\gamma'=1$ and our choice that $\delta'$ is divisible by a high power of $\lambda$ also implies that $\beta' \equiv \gamma'^{-1} \pmod {16m}$ so that
\begin{align*}
\begin{split}
 \leg {i^a\lambda^d}{\beta'}_j=\leg {i^a\lambda^u}{\gamma'}^{-1}_j.
\end{split}
\end{align*}

  We then deduce from the above and \eqref{sumpsi} that when $\lambda \nmid \gamma'$,
\begin{align}
\label{sumpsi10}
\begin{split}
 \phi(s, \psi, m, \kappa_l) =&  \frac{\pi}{s-1} \leg {-\alpha'}{\gamma'}_j  \sum_{0 \leq a \leq 3}2^{u(1-s)} \\
& \hspace*{1cm} \times \sum_{\substack{\gamma_1 \bmod {\lambda^u} } } (-1)^{C(\gamma', -i^a\eta^{-1}+\gamma_1\lambda^{8-u})} \psi(-i^a\eta^{-1}+\gamma_1\lambda^{8-u})  \leg {i^a\lambda^u}{-i^a\eta^{-1}+\gamma_1\lambda^{8-u}}_j \leg {i^a\lambda^u}{\gamma'}^{-1}_j\\
& \hspace*{1cm} \times \sum_{\substack{m_1| M |m^{\infty}_1 \\ M \odd \\ M \text{ is a fourth power}}}\varphi_K(m_1) N(M)^{1-s} \sum_{\substack{(c, m)=1 \\ c \odd }}\varphi_K(c)N(c)^{3-4s} \\
=& \frac{\pi}{s-1} \leg {-\alpha'}{\gamma'}_j  \sum_{0 \leq a \leq 3}2^{u(1-s)} \\
& \hspace*{1cm} \times \sum_{\substack{\gamma_1 \bmod {\lambda^u} } } (-1)^{C(\gamma', -i^a\eta^{-1}+\gamma_1\lambda^{8-u})} \psi(-i^a\eta^{-1}+\gamma_1\lambda^{8-u})  \leg {i^a\lambda^u}{-i^a\eta^{-1}+\gamma_1\lambda^{8-u}}_j \leg {i^a\lambda^u}{\gamma'}^{-1}_j\\
& \hspace*{1cm} \times \frac{\zeta^{(m)}_K(4s-4)}{\zeta^{(m)}_K(4s-3)} \varphi_K(m_1)\prod_{\substack{ \varpi \odd \\ \varpi | m_1}}(N(\varpi)^{4s-4}-1)^{-1}.
\end{split}
\end{align}

  We also deduce from \eqref{psidef}, \eqref{psiexpression} and Lemma \ref{lemchieval} that when $\lambda \nmid \alpha'$,
\begin{align}
\label{sumpsi20}
\begin{split}
  \phi(s, \psi, m, \kappa_l) =& \frac{\pi}{s-1} \leg {\gamma'}{\alpha'}_j  \sum_{\gamma \bmod {16m}} (-1)^{C(\alpha', \gamma)} \psi(\gamma) \sum_{c \equiv -\gamma \alpha' \bmod{16m}}\sum_{\substack{d \bmod {16mc} \\ d \equiv -\gamma \beta' \bmod{16m}}}\leg {d}{c}_jN(c)^{-s} \\
=& \frac{\pi}{s-1} \leg {\gamma'}{\alpha'}_j  \sum_{\gamma \bmod {16m}} (-1)^{C(\alpha', \gamma)} \psi(\gamma) \sum_{c \equiv \gamma \alpha' \bmod{16m}}\sum_{\substack{d \bmod {16mc} \\ d \equiv -\gamma \beta' \bmod{16m}}}\leg {d}{c}_jN(c)^{-s}.
\end{split}
\end{align}

  Note that the condition $c \equiv \gamma \alpha' \pmod{16m}$ implies that $c$ is primary. Thus, we write $c=Mc'$ with $M|m^{\infty}$, $M, c'$ being $\odd$ and $(c', m)=1$ to see that
\begin{align*}
\sum_{\substack{d \bmod {16mc} \\ d \equiv \alpha \bmod{16m}}}\leg {d}{c}_j= &
\begin{cases}
 \leg {\alpha}{M}_jN(M)\varphi_K(c'),  & \text{when $c'$ is a fourth power}, \\
 0, & \text{otherwise}.
\end{cases}
\end{align*}

  We may thus write $c$ as $Mc^4$ with $(c, m)=1$. Suppose that $m=m_1m_2, m_1|\alpha', (\alpha', m_2)=1$. Then the relation $\gamma \alpha' \equiv Mc^4 \pmod {16m}$ implies that $m_1|M$. Moreover, one must have $(M, m_2)=1$ since otherwise the condition $\gamma \alpha' \equiv Mc^4 \pmod {16m}$ implies that $\gamma \equiv 0 \pmod {m_2}$ so that $\leg {\gamma}{M}_j=0$. Thus the condition $\gamma \alpha' \equiv Mc^4 \pmod {16m}$ becomes $\gamma \alpha' \equiv Mc^4 \pmod {16m_2}$. Upon writing $\gamma=16m_2\gamma'+m_1c_1$ with $m_1c_1$ satisfying the condition $\alpha' m_1c_1 \equiv Mc^4 \pmod {16m_2}$ and $\gamma'$ running over a complete residue class modulo $m_1$, as $\psi$ is a ray class character modulo $16$, it follows that the sum over such $\gamma \pmod {16m}$ then involves with a sum over $\sum_{\gamma' \pmod {m_1}}\leg {\gamma'}{M}_j=0$ unless $M$ is a fourth power. We may thus write $M=M'^4$ with $m_1|M'^4$ and notice that the congruence condition satisfied by $\gamma$ now becomes $\gamma \alpha' \equiv M'^4c^4 \pmod {16m_2}$. We write this condition equivalently as $\gamma  \equiv M'^4c^4\alpha'^{-1} \pmod {16m_2}$, where $\alpha'^{-1}$ is the inverse of $\alpha'$ modulo $16m_2$. We may thus write $\gamma=M'^4c^4\alpha'^{-1}m_1C_1+\gamma_116m_2C_2$ with $C_1, C_2$ satisfying the conditions $m_1C_1 \equiv 1 \pmod {16m_2}, 16m_2C_2 \equiv 1 \pmod {m_1}$ and $\gamma_1$ running over a complete residue class modulo $m_1$. Thus, when $\psi^4$ is the principal Hecke character, we have $\psi(\gamma)=\psi(\alpha'^{-1})$. Similarly, we have
\begin{align*}
\begin{split}
(-1)^{C(\alpha', \gamma)}=& (-1)^{C(\alpha', M'^4c^4\alpha'^{-1}m_1C_1+\gamma_116m_2C_2)}=(-1)^{C(\alpha', \alpha'^{-1})},
\end{split}
\end{align*}
 where the last equality holds by noting that $N(M')^4 \equiv 1 \pmod 8$ when $(M', 2)=1$. \newline

 Moreover, as $\alpha'\delta'-\beta'\gamma'=1$, we get that $(\beta', \alpha')=1$, hence $(\beta', M)=1$. Thus $\leg {-\beta'}{M}_j=1$ if $M$ is a fourth power. We then deduce from the above and \eqref{sumpsi20} that when $\lambda \nmid \alpha'$,
\begin{align}
\label{sumpsi1}
\begin{split}
  \phi(s, \psi, m, \kappa_l)
=&  \frac{\pi}{s-1}\leg {\gamma'}{\alpha'}_j  \psi(\alpha'^{-1})(-1)^{C(\alpha', \alpha'^{-1})} \sum_{\substack{m_1| M |m^{\infty}_1 \\ M \odd \\ M \text{ is a fourth power}}}\frac{\varphi_K(m_1)}{N(M)^{s-1}} \sum_{\substack{(c, m)=1 \\ c \odd }} \frac{\varphi_K(c)}{N(c)^{4s-3}} \\
=&  \frac{\pi}{s-1}\leg {\gamma'}{\alpha'}_j  \psi(\alpha'^{-1})(-1)^{C(\alpha', \alpha'^{-1})} \frac{\zeta^{(m)}_K(4s-4)}{\zeta^{(m)}_K(4s-3)}\varphi_K(m_1)\prod_{\substack{ \varpi \odd \\ \varpi | m_1}}(N(\varpi)^{4s-4}-1)^{-1}.
\end{split}
\end{align}

We note that 
  \[ \varphi_K(m_1)\prod_{\substack{ \varpi \odd \\ \varpi | m_1}}(N(\varpi)^{4s-4}-1)^{-1}=1 \]
for $s=5/4$. It follows from the expressions given in \eqref{sumpsi10} and \eqref{sumpsi1} that the relation given in \eqref{Resphirel} holds regardless of which one of the conditions $\lambda \nmid \gamma'$ or $\lambda \nmid \alpha'$ holds. This completes the proof of the proposition.
\end{proof}

   We now deduce from \eqref{Resrel} by comparing the coefficients $\phi(s, \nu, \psi, m;k)$  of the Fourier expansions of $E_k(w, s, \psi, m)$ to infer from \eqref{EiFourier} by taking $N=16, 16m$ there that, for any primary $k$ and any $\nu \in \mathcal O^*_K$,
\begin{align}
\label{Fourierrel}
\begin{split}
  \res_{s=5/4}\phi(s, \nu, \psi, 1;k)=& \Big (\prod_{\substack{ \varpi \odd \\ \varpi | m}}(1+N(\varpi)^{-1})\Big ) \res_{s=5/4} \phi(s, m\nu, \psi, m;k).
\end{split}
\end{align}
  We now define for every primary, square-free $m$,
\begin{align}
\label{phi1def}
\begin{split}
  \phi_1(s, \nu, \psi, m) := \sum_{\substack{k \bmod {16} \\ k \equiv 1 \bmod 4}}\phi(s, \nu, \psi, m;k) \quad \mbox{and} \quad  \phi_2(s, \nu, \psi, m) := \sum_{\substack{k \bmod {16} \\ k \equiv 1+\lambda^3 \bmod 4}}\phi(s, \nu, \psi, m;k).
\end{split}
\end{align}
  We then apply \eqref{2.05} and \eqref{E12exp} to see that
\begin{align}
\label{psiiexp}
\begin{split}
 \phi_1(s, \nu, \psi, m) = \sum_{\substack{c \equiv 1 \bmod{4} \\ (c, m)=1}}\psi(c)\leg {m}{c}_jg_j(\nu, c)N(c)^{-s} \quad \mbox{and} \quad
 \phi_2(s, \nu, \psi, m) = -\sum_{\substack{c \equiv 1+\lambda^3 \bmod{4} \\ (c, m)=1}}\psi(c)\leg {m}{c}_jg_j(\nu, c)N(c)^{-s}.
\end{split}
\end{align}

 We readily deduce from \eqref{Fourierrel} and \eqref{phi1def} the following result that compares the residue at $s=5/4$ of $\phi_i$ for various $i, m$.
\begin{prop}
\label{lempsirelation} With the notation as above, let $\psi$ be any ray class character modulo $16$. Then for any $\nu \in \mathcal O^*_K$ and $i=1,2$, we have
\begin{align*}
\begin{split}
  \res_{s=5/4}\phi_i(s, \nu, \psi, 1)=& \Big (\prod_{\substack{ \varpi \odd \\ \varpi | m}}(1+N(\varpi)^{-1})\Big ) \res_{s=5/4} \phi_i(s, m\nu, \psi, m).
\end{split}
\end{align*}
\end{prop}

  We now recast $\phi_i(s, \nu, \psi, 1)$ in a manner similar to the expressions given in \cite[(9.1), (9.2)]{Suz1} using the fact that $\psi$ is multiplicative and following the arguments from the bottom of \cite[p. 114]{Suz1} to the top of \cite[p. 114]{Suz1} to deduce from Proposition \ref{lempsirelation} that when $m$ is square-free, $\nu \in \mathcal O^*_K$ such that $(\nu, m)=1$, we have
\begin{align*}
\begin{split}
 \res_{s=5/4}  \phi_1(s, \nu, \psi, 1)  \prod_{\substack{ \varpi \odd \\ \varpi | m}}(1+N(\varpi)^{-1}) =& \sum_{\substack{m_1 |m \\ m_1 \equiv 1 \bmod 4}}\frac{\psi(m_1) g_j(\nu, m_1)}{N(m_1)^{5/4}} \res_{s=5/4}  \phi_1(s, m^2_1\nu, \psi, 1) \\
& \hspace*{1cm} +\sum_{\substack{m_1 |m \\ m_1 \equiv 1+\lambda^3 \bmod 4}} \frac{\psi(m_1) g_j(\nu, m_1)}{N(m_1)^{5/4}} \res_{s=5/4}  \phi_2(s, m^2_1\nu, \psi, 1) \\
\res_{s=5/4}  \phi_2(s, \nu, \psi, 1)  \prod_{\substack{ \varpi \odd \\ \varpi | m}}(1+N(\varpi)^{-1})   =& \sum_{\substack{m_1 |m \\ m_1 \equiv 1 \bmod 4}}\frac{\psi(m_1) g_j(\nu, m_1)}{N(m_1)^{5/4}} \res_{s=5/4}  \phi_2(s, m^2_1\nu, \psi, 1) \\
&\hspace*{1cm} -\sum_{\substack{m_1 |m \\ m_1 \equiv 1+\lambda^3 \bmod 4}} \frac{\psi(m_1) g_j(\nu, m_1)}{N(m_1)^{5/4}} \res_{s=5/4}  \phi_1(s, m^2_1\nu, \psi, 1).
\end{split}
\end{align*}

  We induce on the number of primary prime divisors of $m$ to arrive at relations between the residues at $s=5/4$ of $\phi_1, \phi_2$ when $\nu$ is replaced by $\nu m^2$ for a square-free $m$.
\begin{prop}
\label{resrelation} With the notation as above and let $\psi$ be any ray class character modulo $16$. Then for any square-free $m \in \mathcal O_K$ and any $\nu \in \mathcal O^*_K$ satisfying $(\nu, m)=1$, we have
\begin{align*}
\begin{split}
  \res_{s=5/4}\phi_1(s, m^2\nu, \psi, 1)=&\overline{\psi(m)g_j(\nu, m)}N(m)^{-3/4}\res_{s=5/4}\phi_1(s, \nu, \psi, 1), \quad m \equiv 1 \pmod 4, \\
  \res_{s=5/4}\phi_2(s, m^2\nu, \psi, 1)=&\overline{\psi(m)g_j(\nu, m)}N(m)^{-3/4}\res_{s=5/4}\phi_2(s, \nu, \psi, 1), \quad m \equiv 1 \pmod 4, \\
  \res_{s=5/4}\phi_1(s, m^2\nu, \psi, 1)=&-\overline{\psi(m)g_j(\nu, m)}N(m)^{-3/4}\res_{s=5/4}\phi_2(s, \nu, \psi, 1), \quad m \equiv 1+\lambda^3 \pmod 4, \\
  \res_{s=5/4}\phi_2(s, m^2\nu, \psi, 1)=&\overline{\psi(m)g_j(\nu, m)}N(m)^{-3/4}\res_{s=5/4}\phi_1(s, \nu, \psi, 1), \quad m \equiv 1+\lambda^3 \pmod 4.
\end{split}
\end{align*}
\end{prop}

\section{Proof of Proposition \ref{prop: Fbound}}
\label{section6}

\subsection{Analytic behavior of Dirichlet series associated with quartic Gauss sums}
\label{section: smooth Gauss}
   The proof of Proposition \ref{prop: Fbound} demands the establishment of the analytical properties of certain Dirichlet series associated with quartic Gauss sums. To do so, we begin by introducing some notations. Define
\begin{align}
\label{hdef}
h_a (r, s, \psi) = \su{b \odd \\ a\mid b, \; (b,r)=1}  \psi(b) g_{K,j}(r,b)  N(b)^{-s}.
\end{align}

   We also set for $(a,b)=1$,
\begin{align*}
   \rho_a(b)=(-1)^{C(a,b)},
\end{align*}
   where $C(\cdot,\cdot)$ is given in \eqref{Cdef}. It is easy to see that $\psi_a(c)$ is a ray class character modulo $16$. \newline

   For any square-free, non-unit primary $a \in \mathcal O_K$, let $\{\varpi_1, \cdots, \varpi_k \}$ be the set of distinct primary prime divisors of $a$.  We define
\begin{align*}
  P(a)=\prod^{k}_{i=1}\overline{\leg{(a/\prod^{i-1}_{l=1}\varpi_l)^2}{\varpi^3_i}}_j,
\end{align*}
   where, as per convention, the empty product is $1$. As $\leg{\varpi^2_2}{\varpi^3_1}_4=\leg{\varpi_2}{\varpi_1}^2_4$ for two distinct primes $\varpi_1, \varpi_2$, one checks easily by induction on the number of prime divisors of $a$ that $P(a)$ is independent of the order of $\{\varpi_1, \cdots, \varpi_k \}$. \newline

  Our first lemma is taken from \cite[Lemma 2.4]{G&Zhao1} that establishes relations between various Dirichlet series needed in our study.
\begin{lemma}
\label{lem2} Suppose $r, \alpha$ are primary with $\alpha$ square-free.  For any ray class character $\psi$ modulo $16$, we set
\begin{equation*}
  h(r,s,\psi; \alpha)=\sum_{\substack{ n \odd \\(n,\alpha)=1}}\frac {\psi(n)g_{K,j}(r,n)}{N(n)^s}.
\end{equation*}
Moreover, write $r =r_1r^2_2r^3_3r^4_4$ with $r_i, 1 \leq i \leq 4$ being primary and $r_1r_2r_3$ being square-free.  Let
\begin{equation*} 
 h^*(r_1r_2^2r^3_3,s,\psi;r_1)= \sum_{\substack{a \odd \\ a|r_2}}\mu_{K}(a)\psi(a)^3N(a)^{2-3s}\overline{\leg{-r_1(r_2/a)^2r^3_3}{a^3}}_4 P(a) \left( \prod_{\varpi| a}\overline{g}_{K,j}(\varpi) \right)    h(r_1r_2^2r^3_3/a^2,s, \rho_{a^3}\psi; r_1),
\end{equation*}
  where the empty product is $1$.  Then
\begin{align}
\label{2.13}
  h(r_1r^2_2r^3_3,s,\psi;r_1r_2r_3) &=\prod_{\varpi|r_3}(1-\psi(\varpi)^4N(\varpi)^{3-4s})^{-1}h(r_1r_2^2r^3_3,s,\psi; r_1r_2), \\
\label{2.14}
  h(r_1r_2^2r^3_3,s,\psi; r_1r_2) =& \prod_{\varpi|r_2} \left(1-\rho_{\varpi^3}(\varpi)\psi(\varpi)^4N(\varpi)^{2-4s}|g_{K,j}(\varpi)|^2\overline{\leg{-1}{\varpi^3}}_j \right )^{-1}  h^*(r_1r_2^2r^3_3,s,\psi;r_1),  \quad \mbox{and} \\
\label{2.15}
  h(r_1r_2^2r^3_3,s,\psi;r_1) &=\prod_{\varpi|r_1} \left( 1+\psi(\varpi)^2N(\varpi)^{1-2s}g_{K,2}(\varpi)\overline{\leg{r_1r_2^2r^3_3/\varpi}{\varpi^2}}_j \right)^{-1}h(r_1r_2^2r^3_3,s,\psi;1).
\end{align}
\end{lemma}

  Our next lemma, an analogue to \cite[(20)]{H}, estimates the second moment of $h(r,s, \psi; 1)$ integrated along on the right of the line $\Re(s)=1$. 
\begin{lemma}
\label{lem:from-HB}
  Let $r=r_1r^2_2r^3_3$ such that $r_i, 1 \leq i \leq 3$ are primary, square-free and mutually co-prime and $\psi$ be any ray class character modulo $16$. We have for any $\varepsilon>0$, 
\begin{align}
\label{Integralbound}
\begin{split}
\int\limits_{-T}^{T} \left\vert h(r,1+\varepsilon+it, \psi; 1) \right\vert^2 \dif t \ll T^3 N(r)^{1/2+\varepsilon}.
\end{split}
\end{align}
\end{lemma}
\begin{proof}
  Our proof follows from the treatments in the proof of \cite[Lemma 3]{H}.  For $s = \sigma + it$ with $\sigma, t \in R$, we define
\begin{align}
\label{Zdef}
\begin{split}
 Z(r,s; \psi) := & \zeta_K(4s-3)h(r,s, \psi;1),  \quad \Gamma_{\mathbb{C}}(s):= 2 \cdot (2 \pi)^{-s} \Gamma(s), \\
G_{\infty}(s):=&\Gamma_{\mathbb{C}} \Big(s-\frac{3}{4} \Big) \Gamma_{\mathbb{C}} \Big(s-\frac{1}{2} \Big) \Gamma_{\mathbb{C}} \Big(s-\frac{1}{4} \Big) \quad \mbox{and} \quad F(r,s; \psi) := G_{\infty}(s)Z(r,s;\psi). 
\end{split}
\end{align}  

  It is shown in \cite[Theorem 1]{BrubaderBump} that $F(r,s; \psi)$ is holomorphic, except for possible simple poles at $s = 2/3, 5/4$. Moreover, it satisfies 
a functional equation (see \cite[(4.27)]{DDHL}, \cite[Theorem 1]{BrubaderBump}) given by
\begin{equation} 
\label{funceqn}
 F(r,s; \psi) =N(r)^{1-s}  
\cdot \sum_{j} A_{j}(2^{-s})  F(r,2-s; \overline \psi),
\end{equation}
where $j$ runs over a finite number of indexes and $A_{j}(2^{-z})$ is a rational function in $2^{-z}$ with coefficients in $\mathbb{C}$ such that $A_{j}(2^{-z})$ are holomorphic for $\Re(z) \neq 1$, $5/4$. \newline 

 We further define
\begin{align} 
\label{wideZ}
\begin{split}
 \widetilde Z(r,s; \psi)  := \sum_{\psi}|Z(r, s; \overline \psi)|^2 \quad \mbox{and} \quad \widetilde F(r,s; \chi) :=  G_{\infty}(s)^2 \widetilde Z(r,s; \psi).
\end{split}
\end{align}
From \eqref{funceqn}, we deduce that
\begin{align} 
\label{Fbound}
 \widetilde F(r,s; \psi) \ll N(r)^{2-2\sigma}\Big |\widetilde F(r,2-s; \psi)\Big |, \quad \frac 34 < \sigma < \frac 54. 
\end{align}

  Note that by \cite[(9.20)]{DDHL} implies that for $1 \leq T \leq t \leq 2T$, 
\begin{align} 
\label{Gbound}
G_{\infty}(2-s) \ll  T^{3(2-2\sigma)} |G_{\infty}(s)|. 
\end{align}  
 
  It follows from \eqref{Fbound} and \eqref{Gbound} that
\begin{align} 
\label{wideZbound}
 \Big |\widetilde Z(r,s; \psi)\Big | \ll T^{6(2-2\sigma)}N(r)^{2-2\sigma}\Big |\widetilde Z(r,s; \overline \psi)\Big |. 
\end{align}  
 
  We now write for $\sigma$ large enough, 
\begin{align} 
\label{Zseries}
Z(r,s; \psi) =\sum_{\substack{n \odd}}a_n(r,\psi)N(n)^{-s}. 
\end{align}  
   
   We apply the Mellin transform to get that for $1 < \sigma \leq 5/4$, 
\begin{align*} 
 \sum_{\substack{n \odd}} \frac{a_n(r,\psi)}{N(n)^s} e^{-N(n)/X}=\frac 1{2\pi i}\int\limits_{(2)}Z(r,s+w; \psi)X^w\Gamma(w) \dif w. 
\end{align*}     
  We shift the line of integration to $\Re(w) =2-2\sigma$ to encounter a simple pole at $w = 0$ and a possible simple pole at
  $w = 5/4-s$.  This leads to 
\begin{align*} 
\begin{split}
 \sum_{\substack{n \odd}} & \frac{a_n(r,\psi)}{N(n)^s} e^{-N(n)/X} \\
=&\frac 1{2\pi i}\int\limits_{(2-2\sigma)}Z(r,s+w; \psi)X^w\Gamma(w) \dif w  +\res_{w=0}Z(r,s+w; \psi)X^w\Gamma(w)+\res_{w=5/4-s}Z(r,s+w; \psi)X^w\Gamma(w) \\
=&\frac 1{2\pi i}\int\limits_{(2-2\sigma)}Z(r,s+w; \psi)X^w\Gamma(w)\dif w  + Z(r,s; \psi)+O(N(r)^{1/8+\varepsilon}X^{5/4-\sigma}e^{-|t|}), 
\end{split}
\end{align*} 
  where the last estimation above follows from Lemma on p. 200 of \cite{P} (see also \eqref{res54} below) and \eqref{Stirlingratio}. \newline

  We apply \eqref{Zdef} to see that
\begin{align} 
\label{Iexp}
 I(\sigma) :=\int\limits^{2T}_T\Big |Z(r,s;\psi)\Big |^2 \dif t \ll I_1+I_2+N(r)^{1/4+\varepsilon}X^{5/2-2\sigma}e^{-T}, 
\end{align}   
 where
\begin{align} 
\label{I1I2}
\begin{split}
 I_1= \int\limits^{2T}_T\Big |\sum_{\substack{n \odd}}a_n(r, \psi)N(n)^{-\sigma-it}e^{-N(n)/X} \Big |^2 \dif t, \quad \mbox{and} \quad  I_2= \int\limits^{2T}_T\Big |\int\limits_{(2-2\sigma)}Z(r,\sigma+it+w; \psi)X^w\Gamma(w)dw \Big |^2 \dif t.
\end{split}
\end{align}   

  Using the estimation $N(n)/X e^{-2N(n)/X} \ll 1$,  we deduce from a mean-value result of H. L. Montgomery and R. C. Vaughan \cite[Corollary 3]{MV74}  that
\begin{align} 
\label{I1bound}
\begin{split}
  I_1 \ll & \sum_{\substack{n \odd}}(T+N(n))|a_n(r, \psi)|^2N(n)^{-2\sigma}e^{-2N(n)/X} \ll  (T+X)\sum_{\substack{n \odd}}|a_n(r, \psi)|^2N(n)^{-2\sigma}. 
\end{split}
\end{align} 

Now, from \eqref{rel2}, \eqref{Zdef} and \eqref{Zseries} that $a_n(r, \psi)$ is multiplicative such that for any primary $n$, upon writing 
  $n=\prod_{\substack{\varpi |n \\ \varpi \odd}}\varpi^f$ with $f \geq 1$, we have
\begin{align*} 
\begin{split}
  |a_n(r, \psi)|=\prod_{\substack{\varpi |n \\ \varpi \odd}}|a_{\varpi^f}(r, \psi)|, \quad \mbox{with} \quad a_{\varpi^f}(r, \psi) =\sum_{\substack{4g+e =f \\ m, k \geq 0}}g(r,\varpi^e)N(\varpi^{3g}). 
\end{split}
\end{align*}

We now deduce by \eqref{rel4} that if $r=r_1r^2_2r^3_3$ such that $r_i, 1 \leq i \leq 3$, are square-free and mutually co-prime, then $g_j(r, \pi^e) \ll N(\varpi)^{a(e)}$ with $a(0)=0$, $a(1)=1/2$, $a(2)=3/2$, $a(3)=5/2$, $a(4)=3$ and that $g_j(r, \pi^e)=0$ for $e \geq 5$. Thus, we may assume that $0 \leq e \leq 4$.  If $1 \leq f=e+4g$ in this case, $|a_{\varpi^f}(r, \psi)| \leq N(\varpi)^{a(e)+3g} \leq N(\varpi)^{f-1/2}$. As there are at two different ways to write a given $f \geq 1$ in the form $e+4g$ above, we deduce that uniformly for $\sigma>1+\varepsilon$ with $\varepsilon>0$, 
\begin{align*} 
\begin{split}
  \sum_{\substack{n \odd}}|a_n(r, \psi)|^2N(n)^{-2\sigma} \leq \prod_{\varpi \odd}(1+2N(\varpi)^{1-2\sigma}+2N(\varpi)^{3-4\sigma}+N(\varpi)^{5-6\sigma}+\cdots) \ll 1.
\end{split}
\end{align*}     
Thus \eqref{I1bound} becomes
\begin{align} 
\label{I1est}
\begin{split}
  I_1 \ll_{\varepsilon} T+X. 
\end{split}
\end{align} 

  Next, \eqref{Stirlingratio} and Cauchy's inequality give
\begin{align} 
\label{Zsquarebound}
\begin{split}
 \Big |\int\limits_{(2-2\sigma)}Z(r,\sigma+it+w; \psi)X^w\Gamma(w) \dif w \Big |^2 \leq X^{4-4\sigma}I'I'', 
\end{split}
\end{align}   
  where
\begin{align} 
\begin{split}
 I'=  \int\limits_{(2-2\sigma)}|\Gamma(w) \dif w| \ll_{\varepsilon} 1 \quad \mbox{and} \quad
 I''= \int\limits_{(2-2\sigma)}|Z(r,\sigma+it+w; \psi)|^2|\Gamma(w)| \dif w \ll_{\varepsilon} \int\limits^{\infty}_{-\infty}|Z(r,2-\sigma+iy; \psi)|^2e^{-|t-y|} \dif y.
\end{split}
\end{align}   

 Note that we have
\begin{align} 
\label{ebound}
\begin{split}
 \int\limits^{2T}_Te^{-|t-y|} \dif y \ll
\begin{cases}
 1, & y \in [T, 2T] \\
 e^{2T-y}, & y \geq 2T \\
 e^{y-T}, & y \leq T. 
\end{cases}
\end{split}
\end{align}   
 
  It follows from \eqref{I1I2}, \eqref{Zsquarebound}--\eqref{ebound} that
\begin{align} 
\label{I2est}
\begin{split}
 I_2 \ll_{\varepsilon}  X^{4-4\sigma}\int\limits^{2T}_T|Z(r,2-\sigma+iy; \psi)|^2 & \dif y +X^{4-4\sigma}\int\limits^{\infty}_{2T}|Z(r,2-\sigma+iy; \psi)|^2e^{2T-y} \dif y \\
 & +X^{4-4\sigma}\int\limits^T_{-\infty}|Z(r,2-\sigma+iy; \psi)|^2e^{y-T} \dif y. 
\end{split}
\end{align}   
 
The estimate from the proof of Lemma on p. 200 of \cite{P} renders that, for $3/4 <\alpha \leq 3/2$, $\alpha \neq 1$, 
\begin{align} 
\label{Zest}
\begin{split}
 Z(r,\alpha+iy; \psi) \ll N(r)^{(3/2+\varepsilon-\alpha)/2}(1+t^2)^{3(3/2+\varepsilon-\alpha)/2}. 
\end{split}
\end{align}   
 Now that the proof of Lemma on p. 200 of \cite{P} only establishes the above in the range $1 <\alpha \leq 3/2$. However, one applies the functional equation 
given in \eqref{funceqn} and \eqref{Gbound}, together with the observation that the functions $A_{j}(2^{-z})$ appearing in \eqref{fneqnL} are holomorphic for $\Re(z) \neq 1$, $5/4$, to infer that the above estimation continues to hold for $3/4<\alpha$. \newline

  Applying the estimation in \eqref{Zest} to \eqref{I2est} allows us to see that
\begin{align} 
\label{I2est1}
\begin{split}
 I_2 \ll_{\varepsilon} & X^{4-4\sigma}\Big (\int\limits^{2T}_T|Z(r,2-\sigma+iy; \psi)|^2 \dif y+N(r)^{\sigma-1/2+\varepsilon}T^{6(\sigma-1/2)+\varepsilon}\Big ). 
\end{split}
\end{align}     

  We apply \eqref{Iexp}, \eqref{I1est} and \eqref{I2est1} to see that 
\begin{align} 
\label{Iest}
\begin{split}
 I(\sigma)  \ll_{\varepsilon} & X^{4-4\sigma}\Big (\int\limits^{2T}_T|Z(r,2-\sigma+iy; \psi)|^2 \dif y+N(r)^{\sigma-1/2+\varepsilon}T^{6(\sigma-1/2)+\varepsilon}\Big )+T+X \\
  \ll & X^{4-4\sigma}I(2-\sigma)+T+X+N(r)^{1/4+\varepsilon}X^{5/2-2\sigma}e^{-T}. 
\end{split}
\end{align}   

   We now define
\begin{align*} 
 \widetilde I (\sigma) :=\int\limits^{2T}_T\widetilde Z(r,s;\psi) \dif t. 
\end{align*} 
  We then deduce from \eqref{wideZ} and \eqref{Iest} that 
\begin{align} 
\label{wideIret}
\begin{split}
 \widetilde I (\sigma) \ll_{\varepsilon} & X^{4-4\sigma}\Big (\int\limits^{2T}_T \widetilde Z(r,2-\sigma+iy; \psi)^2 \dif y+N(r)^{\sigma-1/2+\varepsilon}T^{6(\sigma-1/2)+\varepsilon}\Big )+T+X+N(r)^{1/4+\varepsilon}X^{5/2-2\sigma}e^{-T} \\
  \ll_{\varepsilon} & X^{4-4\sigma}\widetilde I(2-\sigma)+X^{4-4\sigma}N(r)^{\sigma-1/2+\varepsilon}T^{6(\sigma-1/2)+\varepsilon}+T+X+N(r)^{1/4+\varepsilon}X^{5/2-2\sigma}e^{-T} \\
  \ll_{\varepsilon} & X^{4-4\sigma}T^{12\sigma-12}N(r)^{2\sigma-2}\widetilde I(\sigma)+X^{4-4\sigma}N(r)^{\sigma-1/2+\varepsilon}T^{6(\sigma-1/2)+\varepsilon}+T+X+N(r)^{1/4+\varepsilon}X^{5/2-2\sigma}e^{-T}, 
\end{split}
\end{align}  
  where the last majorant above follows from \eqref{wideZbound}.  \newline
 
  We write the implied constant in \eqref{wideIret} as $C_{\varepsilon}$.  For $\sigma \geq 1 + \varepsilon$, and
\begin{align*} 
\begin{split}
 X = & (2C_{\varepsilon})^{1/(4\varepsilon)}T^{3}N(r)^{1/2},
\end{split}
\end{align*}  
  we get
\begin{align*} 
\begin{split}
  I(\sigma)  \ll_{\varepsilon} \widetilde I(\sigma) \ll_{\varepsilon} X^{4-4\sigma}N(r)^{\sigma-1/2+\varepsilon}T^{6(\sigma-1/2)+\varepsilon}+T+X+N(r)^{1/4+\varepsilon}X^{5/2-2\sigma}e^{-T}  \ll_{\varepsilon} & T^{3}N(r)^{1/2+\varepsilon}.
\end{split}
\end{align*}   
 
    As $|\zeta_K(4s-3)| \gg_{\varepsilon} 1$ for $\sigma \geq 1+\varepsilon$, we conclude that for $T \geq 1$, 
\begin{align*} 
\begin{split}
  \int\limits^{2T}_T\left\vert h(r,1+\varepsilon+it, \psi; 1) \right\vert^2 \dif t  \ll T^3 N(r)^{1/2+\varepsilon}.
\end{split}
\end{align*}   
   We now sum over $T \geq 1$ dyadically and apply \eqref{Zest} to estimate the integral of $\left\vert h(r,1+\varepsilon+it, \psi; 1) \right\vert^2$ on the interval $[-1, 1]$ to deduce \eqref{Integralbound}. This completes the proof of the lemma. 
\end{proof}

  Our last lemma contains the analytic information on the function $h_a (r, s, \psi)$ defined in \eqref{hdef}. 
\begin{lemma}  \label{lem:sizeof_ha}
 With the notations as above, write $r=r_1r_2^2 r_3^3r^4_4$ with $r_i\odd, r_1, r_2, r_3$ square-free and mutually co-prime. Let $r_4^\ast$ be the product of the primes dividing $r_4$ but not $r_1r_2r_3$.  Suppoes $a$ is primary, square-free and $(a, r)=1$. Then the function $h_a (r, s, \psi)$ can be meromorphically continued to the whole complex plane. It is entire for $\Re(s) > 1$ except possibly for a simple pole at $s=5/4$. Let $\varepsilon>0$ and $s =\sigma+it$ satisfying $1+\varepsilon \le \sigma \le 3/2+\varepsilon$ and $|s-5/4|> 1/8$, we have
\begin{equation}
\label{bound-convexity}
h_a (r, s, \psi) \ll N(r)^{(3/2-\sigma)/2+\varepsilon} N(a )^{2-2\sigma+\varepsilon} (1+t^2)^{3/2 \cdot (3/2-\sigma)+\varepsilon}.
\end{equation}
  Also,  the residue at $s=5/4$ satisfies
\begin{align}
\label{res54}
\begin{split}
  \mathop{\res}_{s=5/4} h_a (r, s, \psi)  \ll N(a )^{-1+\varepsilon} N (r)^{1/8+\varepsilon}.
\end{split}
\end{align}
 Furthermore,
\begin{align}
\label{IntegralMeanValuebound4h_a}
\begin{split}
\int\limits_{-T}^{T} | h_a(r,1+\varepsilon+it, \psi) |^2 \dif t \ll T^3 N(a)^{-\varepsilon} N(r)^{1/2+\varepsilon}.
\end{split}
\end{align}
\end{lemma}
\begin{proof}
Recall that if $(r, b)=1$, $g_{K,j}(r,b)=0$ for $b$ not square-free by \eqref{rel1}, \eqref{rel2} and \eqref{rel4}. We thus replace  $b$ in the sum ${h}_a(r,s, \psi)$ by $b=ab'$ with $b'\odd$ and co-prime to $a$ and apply \eqref{rel3}.  This leads to
\begin{align}
\label{sumoverb}
\begin{split}{{h}}_a (r, s, \psi) =&
 \su{b \odd \\ a\mid b, \; (b, r)=1} \frac{\psi(b) g_{K,j}(r,b)}{N(b)^s} = \frac{g_{K,j}(r, a ) \psi(a)}{N(a)^s}  \su{b \odd \\ (b, ar)=1} \frac{\psi(b) (-1)^{C(a, b)}g_{K,j}(a^2r, b)}{N(b)^s}.
\end{split}
\end{align}
 As $(b,r)=1$ in the above sum, we apply \eqref{rel1} to see that $g_{K,j}(a^2r, b) = \overline{\chi_{j,b} (r_4^4)} g_{K,j}(a^2r_1r_2^2r^3_3,b) = g_{K,j}(a^2r_1r_2^2r^3_3,b)$. Let $r_4^\ast$ be the product of the primes dividing $r_4$ but not $r_1r_2r_3$.  It follows that
\begin{align}
\label{sumoverd}
\begin{split}
& \su{b \odd \\ (b, ar)=1} \frac{\psi(b) (-1)^{C(a, b)}g_{K,j}(a^2r, b)}{N(b)^s} = \su{b \odd \\ (b, ar_1r_2^2r^3_3)=1} \frac{\psi(b) (-1)^{C(a, b)} g_{K,j}(a^2r_1r_2^2r^3_3,b)}{N(b)^s} \su{d \odd \\d \mid (b,r_4^\ast)} \mu_K(d)\\
=& \su{d\odd \\ d \mid  r_4^\ast} \frac{\mu_K(d) \psi(d) (-1)^{C(a, d)}}{N(d)^s} \su{b \odd  \\ (bd, ar_1r_2^2r^3_3)=1} \frac{\psi(b) (-1)^{C(a, b)} g_{K,j}(a^2r_1r_2^2r^3_3,bd)}{N(b)^s} \\
=& \su{d\odd \\ d \mid  r_4^\ast} \frac{\mu_K(d) \psi(d) (-1)^{C(a, d)}}{N(d)^s} \su{b \odd \\ (b, adr_1r_2^2r^3_3)=1} \frac{\psi(b)(-1)^{C(a, b)} g_{K,j}(a^2r_1r_2^2r^3_3,bd)}{N(b)^s} \\
=&\su{d\odd \\ d \mid  r_4^\ast} \frac{\mu_K(d) \psi(d)(-1)^{C(a, d)}g_{K,j}(a^2r_1r_2^2r^3_3,d)}{N(d)^s} \su{b \odd  \\ (b, adr_1r_2^2r^3_3)=1} \frac{\psi(b)(-1)^{C(d, b)}(-1)^{C(a, b)} g_{K,j}(a^2d^2r_1r_2^2r^3_3,b)}{N(b)^s} \\
=&\su{d\odd \\ d \mid  r_4^\ast} \frac{\mu_K(d) \psi(d)(-1)^{C(a, d)}g_{K,j}(a^2r_1r_2^2r^3_3,d)}{N(d)^s} h(r_1a^2d^2r_2^2r^3_3,s,\rho_{ad}\psi; adr_1r_2r_3),
\end{split}
\end{align}
using again \eqref{rel3} and observing that for $(a,r)=1$, $(-1)^{C(ad,b)}=(-1)^{C(a,b)}(-1)^{C(d,b)}=\rho_{ad}(b)$.
Now we apply \eqref{2.13}--\eqref{2.15} to write $h(r_1a^2d^2r_2^2r^3_3,s,\rho_{ad}\psi; adr_1r_2r^3)$ as linear combinations over primary $e|adr_2$ of expressions involving $h(r_1a^2d^2r_2^2r^3_3/e^2,s,\Psi; 1)$ whose coefficients are $\ll N(ar)^{\varepsilon}N(e)^{-1/2+\varepsilon}$ for $\Re(s)>1$, using \eqref{gest}. Here $\Psi$ is a certain ray character modulo $16$ depending on $a, d, e, r$ only.  We now apply the Lemma on p. 200 of \cite{P}.  This gives that each $h(r_1a^2d^2r_2^2r^3_3/e^2,s,\Psi; 1)$ can be meromorphically continued to the whole complex plane and thus so can $h_a (r, s, \psi)$.  Moreover, when $1+\varepsilon \le \sigma \le 3/2+\varepsilon$ and $|s-5/4|> 1/8$, we have
\begin{equation*}
h(r_1a^2d^2r_2^2r^3_3/e^2,s,\Psi; 1) \ll N(ra^2)^{ (3/2-\sigma)/2+\varepsilon} (1+t^2)^{3/2 \cdot (3/2-\sigma)+\varepsilon}.
\end{equation*}
  We now sum over $e$ trivially and use \eqref{gest}, \eqref{sumoverb},  \eqref{sumoverd} to further sum over $d$ trivially. Using again \eqref{gest} to bound $g_{K,j}(r, a )$ by $N(a)^{1/2}$, we deduce \eqref{bound-convexity}. \newline

 Next, we apply Proposition \ref{resrelation} by noting from \eqref{GKrel} that $g_{K,j}(k,c)$ and  $g_{j}(k,c)$ differ only by a ray class group character modulo $16$, so that the assertion of Proposition \ref{resrelation} is still valid if one replaces $g_{j}(k,c)$ by  $g_{K,j}(k,c)$ in the definition of $\phi_1, \phi_2$ in \eqref{psiiexp} throughout. Together with the bound $\overline{g_{K,j}(\nu, m)} \ll N(m)^{1/2+\varepsilon}$ for $(v, m)=1$ from \eqref{gest}, we get
\begin{align*}
\begin{split}
 \res_{s=5/4}h(r_1a^2d^2r_2^2r^3_3/e^2,s,\Psi; 1) \ll & N(adr_2)^{-1/4+\varepsilon}N(e)^{1/4+\varepsilon}|\res_{s=5/4}h(r_1r^3_3,s,\Psi; 1)| \\
  \ll & N(ad)^{-1/4+\varepsilon}N(e)^{1/4+\varepsilon} N(r)^{1/8+\varepsilon},
\end{split}
\end{align*}
  where the last estimation above follows from the Lemma on p. 200 of \cite{P} again. Now repeating the arguments that lead to the estimation given in \eqref{bound-convexity} allows one to establish \eqref{res54}. \newline

  Lastly, for \eqref{IntegralMeanValuebound4h_a}, we argue in a manner similar to the proof of \cite[(23)]{DG22}.  Upon using Cauchy's inequality twice,
\begin{align*}
\begin{split}
  |h_a (r, 1+\varepsilon+it, \psi)|^2 &\ll_{\varepsilon} N(a)^{-1-2\varepsilon} N(r)^{\varepsilon}\su{d\odd \\ d \mid  r_4^\ast} \su{e \odd \\ e \mid dar_2} |h(r_1a^2d^2r_2^2r^3_3/e^2,1+\varepsilon+it,\Psi; 1)|^2.
\end{split}
\end{align*}
  We now apply Lemma \ref{lem:from-HB} to deduce that 
\begin{align*}
\begin{split}
 \int\limits_{-T}^{T} | h_a(r,1+\varepsilon+it, \psi) |^2 \dif t
&\ll T^3 N(a)^{-1-2\varepsilon} N(a^2r_1r_2^2r^3_3)^{1/2} \su{d\equiv 1\mod 3\\ d \mid  r_4^\ast} N(d) \su{e \odd \\ e \mid dar_1}  N(e)^{-1} \\
&\ll T^3 N(a)^{-\varepsilon} N(r_1r_2^2r^3_3(r_4^\ast)^2)^{1/2+2\varepsilon} \ll T^3 N(a)^{-\varepsilon} N(r)^{1/2+2\varepsilon}. 
\end{split}
\end{align*}
 This establishes \eqref{IntegralMeanValuebound4h_a} and hence completes the proof of the lemma. 
\end{proof}

\subsection{Completion of the proof of Proposition \ref{prop: Fbound}}

  We may assume that $z-1/2$ is a positive rational integer to deduce from \cite[(24)]{DDHL} that for $\sigma_1 = 1+ \varepsilon$,
\begin{align}
\label{FwithPerron1}
\begin{split}
F_a(z,r,\psi) = \frac{1}{2\pi i} \int\limits_{\sigma_1 - i T}^{\sigma_1 + i T} {h}_a (r, s+{\textstyle \frac12}, \psi) \frac{z^s}{s} \dif s +O(T^{-1} z^{1+\varepsilon} N(a)^{-1}).
\end{split}
\end{align}
 We shift the contour of the integral in \eqref{FwithPerron1} to $\Re s = 1/2+\varepsilon$, encountering by Lemma \ref{lem:sizeof_ha} a possible simple pole of $h_a(r,s, \psi)$ at $s=5/4$, to get
\begin{align*}
\begin{split}
\frac{1}{2\pi i} & \int\limits_{\sigma_1 - i T}^{\sigma_1 + i T} {h}_a (r, s+{\textstyle \frac12}, \psi) \frac{z^s}{s} \dif s \\
&  = \frac {4z^{3/4}}3 \mathop{\res}_{s=5/4}  {h}_a (r, s, \psi) + \frac{1}{2\pi i} \biggl(\int\limits_{\sigma_1 - i T}^{\sigma_1-1/2 - i T} + \int\limits_{\sigma_1-1/2 + i T}^{\sigma_1 + i T} + \int\limits_{\sigma_1-1/2 - i T}^{\sigma_1-1/2 + i T} \biggr) {h}_a (r, s+{\textstyle \frac12}, \psi) \frac{z^s}{s} \dif s.
\end{split}
\end{align*}
 We further apply the estimation \eqref{bound-convexity} to see that
\begin{align}
\label{FwithPerron2}
\begin{split}
\biggl( \int\limits_{\sigma_1 - i T}^{\sigma_1-1/2 - i T} + \int\limits_{\sigma_1-1/2 + i T}^{\sigma_1 + i T} \biggr) {h}_a (r, s+{\textstyle \frac12}, \psi) \frac{z^s}{s} \dif s \ll & N(ar )^{\varepsilon}\int\limits_{\sigma_1-1/2}^{\sigma_1} T^{3(\sigma_1-\sigma)-1} N(r)^{\frac 12 (\sigma_1-\sigma)} N(a )^{1 -2\sigma}  z^\sigma \dif \sigma \\
 \ll & N(ar)^{\varepsilon} \bigl(T^{-1} N(a )^{-1-\varepsilon}  z^{1+\varepsilon} + T^{1/2}N(r)^{1/4} z^{1/2+\varepsilon}\bigr).
\end{split}
\end{align}

  We conclude from \eqref{res54}, \eqref{FwithPerron1} and \eqref{FwithPerron2} that
\begin{align}
\begin{split}
\label{sumFbound}
 F_a(z,r,\psi)  \ll  z^{1/2+\varepsilon} & \int\limits_{-T}^{T}  \Big| \frac{h_a(r, \sigma_1+it,\psi) }{\sigma_1+it} \dif t\Big| + z^{3/4} N(a )^{-1+\varepsilon} N (r)^{1/8+\varepsilon}
\\ 
& +N(ar)^{\varepsilon} \bigl(T^{-1} N(a )^{-1-\varepsilon}  z^{1+\varepsilon} + T^{1/2}N(r)^{1/4} z^{1/2+\varepsilon}\bigr).
\end{split}
\end{align}

By the mean value estimate \eqref{IntegralMeanValuebound4h_a} of Lemma \ref{lem:sizeof_ha}
 and Cauchy's inequality, we obtain
\begin{align}
\label{FwithPerron3}
\begin{split}
\int\limits_{-T}^{T} \left| h_a(r,\sigma_1+it,\psi) \right| \dif t \ll T^2 N(a)^{-\varepsilon} N(r)^{1/4+\varepsilon},
\end{split}
\end{align}
 so that integration by parts yields
\begin{align*}
\begin{split}
 \int\limits_{-T}^{T}  \Big| \frac{h_a(r,\sigma_1+it,\psi) }{\sigma_1+it} \dif t \Big| \ll T N(a)^{-\varepsilon} N(r)^{1/4+\varepsilon}. 
\end{split}
\end{align*}

 We deduce from \eqref{sumFbound} and \eqref{FwithPerron3} that
\begin{align*}
\begin{split}
 F_a(z,r,\psi) \ll  z^{1/2+\varepsilon}T N(a)^{-\varepsilon} N(r)^{1/4+\varepsilon}+ z^{3/4} N(a )^{-1+\varepsilon} N (r)^{1/8+\varepsilon}
+N(ar)^{\varepsilon} \bigl(T^{-1} N(a )^{-1-\varepsilon}  z^{1+\varepsilon} + T^{1/2}N(r)^{1/4} z^{1/2+\varepsilon}\bigr).
\end{split}
\end{align*}
  Note that our assumption on $z$ implies that $z^{1/4}N(a)^{-1/2}N(r)^{-1/8} \geq 1$.  Set $T$ to be this value in the above estimation and we deduce readily
  the bound in \eqref{Bound4F}. This completes the proof of Proposition \ref{prop: Fbound}.

\section{Proof of Proposition \ref{lemma:laundrylist}}
\label{section-Patterson}

    A key step involved in the proof of Proposition \ref{lemma:laundrylist} is the following bound on a Dirichlet polynomial formed from quartic Gauss sums, which is analogue to the cubic case established in \cite[Theorem 4.4]{DG22}.
\begin{prop}
\label{lemg3} Let $\psi$ be any ray class character modulo $16$ and $r$ be any primary element in $\mathcal O_K$.
We have for real $Z>1$ and any $\varepsilon>0$,
\begin{align}
\label{glambdabound}
\begin{split}
  H_Z(r;\psi) := \sum_{\substack {\varpi \odd \\ (\varpi, r)=1 \\ N(\varpi) \leq Z}} \frac {\Lambda_K(\varpi)\psi(\varpi) g_{K,j}(r, \varpi) }{\sqrt{N(\varpi)}}  \ll & \min (Z^{1+\varepsilon},  N(r)^{1/16+\varepsilon}Z^{7/8+\varepsilon}).
\end{split}
\end{align}
\end{prop}

\subsection{Vaughan's Identity}

   The proof of Proposition \ref{lemg3} follows the proof of \cite[Theorem 1]{H} as well as the proof of \cite[Theorem 4.4]{DG22}. In the sequel, we use $a \sim A$ to mean that $A<a \leq 2A$ for $a \in \mathbb{N}$ and $A>0$.  Recall also that the expression $\tilde g_\psi (r,c)$ is defined in \eqref{gtildedef}. 
  We start with the following result that is a slight modification of \cite[Lemma 5.1]{DG22}, which comes from an application of Vaughan's identity (see \cite[\S 24]{Da} or \cite[p. 101]{H}).
\begin{lemma}
\label{V-Identity}
  Let $\psi$ be any ray class character modulo $16$ and $r$ be any primary element in $\mathcal O_K$. Define
\begin{align*}
 \Sigma_{k}(Z, r,u) = \sum_{a,b,c} \Lambda_K(a) \mu_K(b) \tilde{g}_\psi(r,abc),\qquad (0 \le k \le 4),
\end{align*}
where $\tilde{g}_\psi$ is defined in \eqref{gtildedef}, each sum runs over $a,b,c \in \mathcal O_K$ which are square-free with $a,b,c$ primary, $N(abc) \sim Z$,  $(r, abc)=1$, and subject to the conditions
\begin{eqnarray*}
&N(bc) \le u, \quad\quad &k=0\\
&N(b) \le u, \quad\quad &k=1\\
&N(ab) \le u, \quad\quad &k=2'\\
&N(a), N(b) \le u < N(ab) \quad\quad &k=2''\\
&N(b) \le u < N(a), N(bc) \quad\quad &k=3\\
&N(a)< N(bc) \le u, \quad\quad &k=4.\\
\end{eqnarray*}
  Then
\begin{align*}
 \Sigma_0(Z,r,u) = \Sigma_{1} (Z,r,u)- \Sigma_{2'}(Z,r,u) - \Sigma_{2''} (Z,r,u) - \Sigma_3 (Z,r,u)+\Sigma_4(Z,r,u).
\end{align*}
Furthermore,
$$
\Sigma_0(Z,r,u) = H_{2Z}(r;\psi) - H_Z(r;\psi),$$
and if we suppose that $1 \le u \le Z^{1/2}$, then $\Sigma_4(Z,r,u)=0$.
\end{lemma}

   The sums $\Sigma_k$ appearing in the above lemma are now divided into the so-called Type I sums ($\Sigma_{1}$ and $\Sigma_{2'}$) and Type II sums ($\Sigma_{2''}$ and $\Sigma_{3}$). An easy adaptation of the proof of \cite[Lemma 6.2]{DDHL} allows us to bound $\Sigma_1$ and $\Sigma_{2'}$.
\begin{lemma}
\label{type1init}
Let $Z \geq 1$, $1 \leq u<Z^{1/2}$. We have
\begin{align*}
 |\Sigma_{1}(Z, r,u) |  & \leq \sum_{\substack{ N(\alpha) \leq u \\ \alpha \odd }} \mu^2_K(\alpha)\Big(|F_{\alpha}(2Z,r, \psi)| \log(2Z)+|F_{\alpha}(Z,r, \psi)| \log Z
+\int\limits_{Z}^{2Z} |F_{\alpha}(x,r, \psi) | \frac{\dif x}{x} \Big),     \\
| \Sigma_{2'}(Z, r,u) |
& \leq  (\log u )\sum_{\substack{N(\alpha) \leq u \\ \alpha \odd}} \mu^2_K(\alpha)(|F_{\alpha}(Z,r, \psi)|+|F_{\alpha}(2Z,r, \psi)| ),
\end{align*}
 where $F_a(z,r,\psi)$ is given in \eqref{Fadef}. 
\end{lemma}

\subsection{Type II sums}
\label{sub-type2}

 We now estimate from above the Type II sums.  
\begin{prop} \label{type2prop}
Let $1 \leq u<Z^{1/2}$. We have for any $\varepsilon>0$, 
\begin{align}
\label{Sigma23est1}
\begin{split}
 \Sigma_{2''}(Z,r,u), \hspace{0.1cm} \Sigma_3 (Z,r,u) \ll_{\varepsilon} Z^{\varepsilon}(Z^{1/2} u+Z u^{-1/2}).
\end{split}
\end{align}
\end{prop}
\begin{proof}
  As the proofs are similar, we only consider the case for $\Sigma_{2''}(Z,r,u)$ here. 
  As $(r, abc)=1$, we see from \eqref{rel2} and \eqref{rel4} that $\tilde{g}_\psi(r,abc)=0$ unless $\mu^2_K(abc)=1$. It follows from this,  \eqref{bilaw}, \eqref{Cdef} and \eqref{rel2} that $\tilde{g}_\psi(r,abc) = (-1)^{C(ab,c)}\Big(\frac{c}{ab} \Big)_2 \tilde{g}_\psi(r,ab) \tilde{g}_\psi(r,c)$ whenever $(ab,c) = 1$, where $C(\cdot,\cdot)$ is given in \eqref{Cdef}. Thus we have
\begin{equation}
\label{2primeintermed}
\Sigma_{2''}(Z,r,u)=\sum_{\substack{v,w \odd \\ u<N(v) \leq u^2 \\  N(wv) \sim Z }}
(-1)^{C(v,w)} A(v) B(w) \Big(\frac{w}{v} \Big)_2,
\end{equation}
 where
\begin{align*}
\begin{split}
A(v) := \widetilde{g}_\psi(r, v)  \sum_{\substack{ab=v \\ N(a), \ N(b) \leq u \\ a,b \odd } } \Lambda_K(a) \mu_K(b) \quad \mbox{and} \quad B(w) := \widetilde{g}_\psi(r, w).
\end{split}
\end{align*}
 Similarly,
\begin{equation*} 
\Sigma_3 (Z,r,u) =\sum_{\substack{v,w \odd \\ u<N(v),\ N(w) \leq 2Z/u  }}
(-1)^{C(v,w)} C(v) D(w) \Big(\frac{w}{v} \Big)_2 ,
\end{equation*}
where
\begin{align*}
\begin{split}
C(v):= \Lambda_K(v) \widetilde{g}_\psi(r, v) \quad \mbox{and} \quad D(w):= \widetilde{g}_\psi(r, w)\sum_{\substack{bc=w \\ N(b) \leq u }} \mu_K(b).
\end{split}
\end{align*}

Note also that
\begin{align}
\label{ABCDbound}
 A(v), B(w), C(v), D(w) \ll_\varepsilon Z^\varepsilon,
 \end{align}
for all relevant $v,w$ and that the functions are supported on square-free primary elements in $\mathcal O_K$ as a consequence of $\mu^2_K(abc)=1$. \newline

  We now estimate \eqref{2primeintermed} by partitioning dyadically $N(v) \sim V$ and $N(w) \sim W$, where
\begin{equation} \label{supportcond}
u/2 \leq V \leq u^2, \quad \text{and} \quad Z/4 \leq VW \leq 2Z.
\end{equation}

  Moreover, note that for any primary element $n \in \mathcal O_K$, we have $n \equiv 1 \pmod 4$ or $n \equiv 1+\lambda^3 \pmod 4$. It is then easy to see that $N(n) \pmod 4$ depends only on $n \pmod 4$. Thus, we further divide primary $v, w$ according to their residues modulo $4$ so that $(-1)^{C(v,w)}$ is a constant for $v, w$ lying in fixed residue classes modulo $4$. It follows that
\begin{equation}
\label{Sigma2dyadic}
\Big |\Sigma_{2''}(Z,r,u) \Big | \leq \sum_{V, W \text{dyadic}}\sum_{\substack{\eta,\gamma \in \{1,1+\lambda^3 \} \bmod 4 }} \Big |\sum_{\substack{v,w \odd \\ N(v) \sim V, \ N(w) \sim W  \\  N(wv)  \sim Z \\ v \equiv \eta \bmod 4 \\ w \equiv \gamma \bmod 4 }}A(v) B(w) \Big(\frac{w}{v} \Big)_2\Big |.
\end{equation}

Perron's formula, as given in \cite[Theorem 5.2, Corollary 5.3]{MVa1}, gives that
\begin{align}
\label{Sigma2est}
\begin{split}
  \sum_{\substack{v,w \odd \\ N(v) \sim V, N(w) \sim W  \\  N(wv)  \sim Z \\ v \equiv \eta \bmod 4 \\ w \equiv \gamma \bmod 4 }}A(v) B(w) \Big(\frac{w}{v} \Big)_2=&\frac 1{2\pi i}\int\limits^{1+\varepsilon+Zi}_{1+\varepsilon-Zi}F_{\eta, \gamma}(s)\frac {(2Z)^s-Z^s}{s}\dif s+R_1+R_2,
\end{split}
 \end{align}
  where
\begin{align*}
\begin{split}
   F_{\eta, \gamma}(s) := \sum_{\substack{v,w \odd \\ N(v) \sim V, \ N(w) \sim W  \\ v \equiv \eta \bmod 4 \\ w \equiv \gamma \bmod 4 }} & \frac{A(v) B(w)}{N(vw)^s} \Big(\frac{w}{v} \Big)_2  , \quad
  R_1 \ll \sum_{\substack{v,w \odd \\ N(v) \sim V, \ N(w) \sim W  \\  N(wv)=n \neq 2Z  \\ v \equiv \eta \bmod 4 \\ w \equiv \gamma \bmod 4 }}\Big | A(v) B(w) \Big(\frac{w}{v} \Big)_2\Big |\min \left( 1, \frac {2Z}{Z|n-2Z|} \right) \quad \mbox{and} \\
   R_2 \ll & \frac {4^{1+\varepsilon}+(2Z)^{1+\varepsilon}}{Z}\sum_{\substack{v,w \odd \\ N(v) \sim V, N(w) \sim W  \\  N(wv)=n  \\ v \equiv \eta \bmod 4 \\ w \equiv \gamma \bmod 4 }} \Big | A(v) B(w) \Big(\frac{w}{v} \Big)_2\Big |n^{-1-\varepsilon} .
\end{split}
\end{align*}

Let $\mathcal{D}_K(n)$ denote the number of distinct prime factors of $n \in \mathcal O_K$.  We have (which can be derived similar to the proof of the classical case over $\mq$ given in\cite[Theorem 2.10]{MVa1})
\begin{align*}
   \mathcal{D}_K(h) \ll \frac {\log N(h)}{\log \log N(h)}, \quad \mbox{for} \quad N(h) \geq 3.
\end{align*}

Now \eqref{ABCDbound} and the above lead to
\begin{align*}
  \sum_{\substack{v,w \odd \\ N(v) \sim V, \ N(w) \sim W  \\  N(wv) =n \\ v \equiv \eta \bmod 4 \\ w \equiv \gamma \bmod 4 }} \Big | A(v) B(w) \Big(\frac{w}{v} \Big)_2 \Big | \ll Z^{\varepsilon}\mathcal{D}_K(n) \ll Z^{\varepsilon},
 \end{align*}
  where the last estimation above follows by observing that $N(n) \ll Z$. We use 1 inside the minimum appearing in the majorant for $R_1$
for the $n$ closest to $2Z$ and the other term in the minimum all other values of $n$.  Hence we arrive at
\begin{align}
\label{R1est}
   R_1 \ll Z^{\varepsilon}\sum_{\substack{Z < n \leq 4Z \\ n \neq 2Z}}\frac 1n \ll Z^{\varepsilon}.
\end{align}

  Similarly, we have
\begin{align}
\label{R2est}
  R_2 \ll Z^{\varepsilon}.
 \end{align}

   It follows from \eqref{Sigma2est}, \eqref{R1est} and \eqref{R2est} that
\begin{align}
\label{Sigma2est1}
\begin{split}
  \sum_{\substack{v,w \odd \\ N(v) \sim V, \ N(w) \sim W  \\  N(wv)  \sim Z \\ v \equiv \eta \bmod 4 \\ w \equiv \gamma \bmod 4 }} & A(v) B(w) \Big(\frac{w}{v} \Big)_2=\frac 1{2\pi i}\int\limits^{1+\varepsilon+Zi}_{1+\varepsilon-Zi}F_{\eta, \gamma}(s)\frac {(2Z)^s-Z^s}{s}\dif s+O(Z^{\varepsilon}) \\
  \ll & Z^{\varepsilon}+Z^{1+\varepsilon}\max_{t}|F_{\eta, \gamma}(1+\varepsilon+it)|\int\limits^Z_{-Z}\frac {\dif t}{1+|t|} \ll Z^{\varepsilon}+Z^{1+\varepsilon}\max_{t}|F_{\eta, \gamma}(1+\varepsilon+it)|. 
\end{split}
 \end{align}

Supposing that $\displaystyle \max_{t}|F_{\eta, \gamma}(1+\varepsilon+it)|$ is attained at some $t=t_0$, we define
\begin{align*}
\widetilde{A}(v):=V A(v) N(v)^{-1-\varepsilon-it_0} \quad \text{and} \quad \widetilde{B}(w):=W B(w) N(w)^{-1-\varepsilon-it_0}.
\end{align*}
  Then we deduce from \eqref{ABCDbound} that
\begin{equation} \label{tildebd}
 \widetilde{A}(v), \widetilde{B}(w) \ll Z^{\varepsilon}.
\end{equation}

  It follows from \eqref{Sigma2est1} and the observation that $VW \sim Z$ that we have
\begin{align*}
\begin{split}
  \sum_{\substack{v,w \odd \\ N(v) \sim V, \ N(w) \sim W  \\  N(wv)  \sim Z \\ v \equiv \eta \bmod 4 \\ w \equiv \gamma \bmod 4 }} A(v) B(w) \Big(\frac{w}{v} \Big)_2
  \ll & Z^{\varepsilon}+Z^{\varepsilon}\Big |\sum_{\substack{N(v) \sim V \\ v \equiv \eta \bmod 4 }} \widetilde{A}(v) \sum_{\substack{ N(w) \sim W \\ w \equiv \gamma \bmod 4 }} \widetilde{B}(w)
\Big(\frac{w}{v} \Big)_2\Big |.
\end{split}
 \end{align*}

  We now apply Cauchy's inequality, the quadratic large sieve, Theorem \ref{quadsieve}, (upon noting that  $\widetilde{A}(v)$ and $\widetilde{B}(w)$
are supported on square-frees) and \eqref{tildebd} to obtain that
\begin{align}
\label{sqbd}
\begin{split}
 \Big | \sum_{\substack{ N(v) \sim V \\ v \equiv \eta \bmod 4 }} \widetilde{A}(v) \sum_{\substack{ N(w) \sim W \\ w \equiv \gamma \bmod 4} } \widetilde{B}(w)
\Big( \frac{w}{v} \Big)_2  \Big |^2  & \leq \Big( \sum_{\substack{ N(v) \sim V \\ v \equiv \eta \bmod{4}}} |\widetilde{A}(v)|^2 \Big) \Big(  \sum_{\substack{ N(v) \sim V \\ v \equiv \eta \bmod{4}}}
\mu^2_K(v) \Big | \sum_{\substack{ N(w) \sim W \\ w \equiv \gamma \bmod{4} }} \mu^2_K(w)\widetilde{B}(w) \Big( \frac{w}{v} \Big)_2 \Big |^2  \Big)  \\
& \ll (VW)^{1+\varepsilon}(V+W).
\end{split}
\end{align}
Inserting the bound \eqref{sqbd} into \eqref{Sigma2est1} leads to
\begin{equation} \label{dyadbd}
 \sum_{\substack{v,w \odd \\ N(v) \sim V, \ N(w) \sim W  \\  N(wv) \sim Z \\ v \equiv \eta \bmod 4 \\ w \equiv \gamma \bmod 4 }} A(v) B(w) \Big(\frac{w}{v} \Big)_2 \ll X^{\varepsilon} (VW)^{1/2} (V^{1/2}+W^{1/2}).
\end{equation}
 We further substitute \eqref{dyadbd} into \eqref{Sigma2dyadic} and apply \eqref{supportcond} to deduce the estimation
for $\Sigma_{2''}(Z,r,u)$ given in \eqref{Sigma23est1}. This completes the proof of the Proposition.
\end{proof}

\subsection{Proof of Proposition \ref{lemg3}}
\label{section-PTT1}

    Upon summing trivially using $|g_{K,j}(r, \varpi)| \leq  N(\varpi)^{1/2}$ for any $(r, \varpi)=1$ by \eqref{gest} implies that $|H_Z(r;\psi)| \ll Z^{1+\varepsilon} \ll Z^{7/8+\varepsilon}N(r)^{1/16+\varepsilon}$ when $Z < N(r)^{1/2}$. We may thus assume that $Z \geq N(r)^{1/2}$ and set $u = Z^{1/4} N(r)^{-1/8}$ to see that $u \geq 1$ in this case. Moreover,  the condition  $N(a) \leq z^{1/4}N(r)^{-1/8}$ in Proposition \ref{prop: Fbound} is also satisfied. Thus, we apply Lemma \ref{type1init}, Proposition \ref{type2prop} and Proposition \ref{prop: Fbound} to see that when $Z \geq N(r)^{1/2}$, 
\begin{align*}
\begin{split}
& H_{2Z}(r;\psi) -H_Z(r;\psi) \ll u^{1/2+\varepsilon}Z^{3/4+\varepsilon}N(r)^{1/8+\varepsilon}+Z^{\varepsilon}(Z^{1/2} u+Z u^{-1/2}) \ll Z^{7/8+\varepsilon}N(r)^{1/16+\varepsilon}.
\end{split}
\end{align*}
 Summing $Z$  dyadically, we arrive at the estimation \eqref{glambdabound}. This completes the proof of Proposition \ref{lemg3}. 

\subsection{Completion of the proof of Proposition \ref{lemma:laundrylist}}
\label{section-PTT}

  With the help of Proposition \ref{lemg3}, we now proceed to the proof of Proposition \ref{lemma:laundrylist}. We recast the function $H(r,s;\psi)$ in \eqref{h} as
\begin{align}
\label{hdecomp}
   H(r,s;\psi)=\sum_{\substack{\varpi \text{ primary} \\ (\varpi,r)=1 \\ N(\varpi) \leq N(r)^{1/2} }}\frac {\Lambda_K(\varpi)\psi(\varpi)g_{K,j}(r,\varpi)}{N(\varpi)^s}+\sum_{\substack{\varpi \text{ primary} \\ (\varpi,r)=1 \\ N(\varpi) > N(r)^{1/2} }}\frac {\Lambda_K(\varpi)\psi(\varpi)g_{K, j}(r,\varpi)}{N(\varpi)^s}.
\end{align}
   We bound the first sum above trivially using the estimation $g_{K, j}(r,\varpi) \ll N(\varpi)^{1/2}$ for any $(r, \varpi)=1$ by \eqref{gest} to see that
for $\Re(s) \leq 3/2$ and any $\varepsilon>0$,  it is bounded by the right-hand side expression in \eqref{hbound}. Meanwhile, we estimate the second sum in \eqref{hdecomp} using \eqref{glambdabound} and partial summation.  This gives that for $s > 3/2-1/8$ and any $\varepsilon>0$, it is also bounded by the right-hand side expression in \eqref{hbound}. The estimation given in \eqref{hbound} now follows from the above, completing the proof of Proposition \ref{lemma:laundrylist}.

\section{Proof of Theorems \ref{Theorem for all characters}--\ref{Theorem one-level densityQ}}
\label{sec: mainthm}

   As the proofs of Theorems \ref{Theorem for all characters}--\ref{Theorem one-level densityQ} are similar to the ones for the cubic case established in 
\cite{G&Zhao17}, we shall only give a sketch of the proof of Theorem \ref{Theorem for all characters} here. For the proof of Theorems \ref{Theorem for all charactersQ}--\ref{Theorem one-level densityQ}, we only point out here that one needs to use the result from \cite[Lemma 2.2]{G&Zhao7} that the primitive quartic Dirichlet characters of prime conductor $p$ co-prime to $2$ such that their squares remain primitive are induced by
$\chi_{j,\varpi}$ for some prime $\varpi \in \mathcal O_K$ such that $N(\varpi) = p$.

\subsection{Proof of Theorem \ref{Theorem for all characters}}

The Mellin inversion gives rise to
\begin{equation}
\label{Integral for all characters}
	\sumstar_{\substack{\varpi}}\frac {\Lambda_{K}(\varpi)L(\tfrac 12+\alpha, \chi_{j, \varpi})}{L(\tfrac 12+\beta, \chi_{j, \varpi})}w \bfrac {N(\varpi)}X=\frac1{2\pi i}\int\limits_{(c)}A_{K,j}\lz s,\tfrac12+\alpha, \tfrac12+\beta\pz X^s\widehat w(s) \dif s,
\end{equation} 
for $\Re(s), \Re(w)$ and $\Re(z)$ large enough, where
\begin{align}
 \label{Aswzexp}
\begin{split}
A_{K,j}(s,w,z)= :& \sumstar_{\varpi}\frac{\Lambda_{K}(\varpi) L(w,  \chi_{j, \varpi})}{N(\varpi)^sL(z,  \chi_{j, \varpi})}=\sumstar_{ n }\frac{\Lambda_{K}(n) L(w, \chi_{j, n})}{N(n)^sL(z,  \chi_{j, n})}-\sum_{i \geq 2}\sumstar_{\substack{\varpi^i \\ \varpi \odd}}\frac{\Lambda_{K}(\varpi^i) L(w, \chi_{j,\varpi^i})}{N(\varpi)^{is}L(z,  \chi_{j, \varpi^i})} \\
:= & A_{K,j,1}(s,w,z)-A_{K,j,2}(s,w,z),
\end{split}
\end{align}
 and
\begin{align*}
     \widehat{w}(s) =\int\limits^{\infty}_0w(t)t^s\frac {\dif t}{t}.
\end{align*}

  We write $m=\lambda^{r_1}m', k=\lambda^{r_2}k'$ with $r_1, r_2 \geq 0$, $m', k'$ primary and apply the ray class characters to detect the condition that $n \equiv 1 \pmod {16}$ to arrive at
\begin{align}
\label{Aswmain-1}
\begin{split}
 A_{K,j,1}(s,w,z)= & \frac{1}{\# h_{(16)}} \sum_{\substack{r_1,r_2 \geq 0  }}\frac{\mu_K(\lambda^{r_2})}{2^{r_1w}2^{r_2z}} \sum_{\psi \bmod {16}} \ \sum_{\substack{m,k \odd }}\frac{\mu_K(k)}{N(m)^wN(k)^z}\sum_{ n \odd }\frac{\Lambda_{K}(n)\psi(n)\chi^{(mk)}_j(n)}{N(n)^s} \\
=&  \frac{1}{\# h_{(16)}} \sum_{\substack{r_1,r_2 \geq 0  }}\frac{\mu_K(\lambda^{r_2})}{2^{r_1w}2^{r_2z}}\sum_{\psi \bmod {16}} \ \sum_{\substack{m,k \odd }}\frac{\mu_K(k)}{N(m)^wN(k)^z}\cdot -\frac {L'(s, \psi \chi^{(mk)})}{L(s, \psi \chi^{(mk)})}.
\end{split}
\end{align}

 Note that $A_{K,j,2}(s,w,z)$ is easily seen by \eqref{Lderboundgen} to converge when $\Re(s)>1/2, \Re(w) \geq 1/2, \Re(z)>1/2$ under GRH. We deduce from this, \eqref{Lderboundgen} and \eqref{Aswmain-1} that under GRH, other than a simple pole at $s=1$ when $mk$ is a perfect fourth power and $\psi$ is the principal character, $A_{K,j}(s,w,z)$ is convergent in the region when $\Re(s)>1/2, \Re(w)>1, \Re(z)>1$. Also, we have
\begin{align}
\label{Ress=1}
\begin{split}
  \res_{s=1}A_{K,j}(s,w,z)
=&  \frac{1}{\# h_{(16)}} \frac{1-2^{-z}}{1-2^{-w}} \frac {\zeta^{(j)}_K(4w)}{\zeta^{(j)}_K(3w+z)}.
\end{split}
\end{align}

 Moreover, we observe that the sum over $\varpi$ in \eqref{Integral for all characters} is convergent when $\Re(s)>1, \Re(w)\geq 1/2, \Re(z)>1/2$ under GRH by \eqref{Lderboundgen}. It therefore follows the above and Theorem \ref{Bochner} that  $A_{K,j}(s,w,z)$ is convergent in the region
\begin{equation*}
		S_{1} :=\left\{ (s,w,z): \ \Re(s)> \frac 12, \ \Re(w) \geq \frac 12, \ \Re(z)>\frac 12, \ \Re(s+w)> \frac 32, \ \Re(s+z)> \frac 32 \right\}.
\end{equation*}

 We then apply the functional equation \eqref{fneqnL} for $L(w, \chi_{j, \varpi})$ to deduce from \eqref{Aswzexp} that
\begin{align*}
\begin{split}
A_{K,j}& (s,w, z)= \sumstar_{\varpi}\frac{\Lambda_{K}(\varpi) L(w,  \chi_{j, \varpi})}{N(\varpi)^sL(z,  \chi_{j, \varpi})} = |D_K|^{1/2-w}(2\pi)^{2w-1}\frac {\Gamma(1-w)}{\Gamma (w)} \sumstar_{\varpi} \frac{\Lambda_{K}(\varpi)g_{j}(\varpi) L(1-w,  \overline \chi_{j, \varpi})}{N(\varpi)^{s+w}L(z,  \chi_{j, \varpi})} \\
=& |D_K|^{1/2-w}(2\pi)^{2w-1}\frac {\Gamma(1-w)}{\Gamma (w)} \sum_{ m,k \odd} \frac {\mu_K(k)}{N(m)^{1-w}N(k)^{z}} \frac{1}{\# h_{(16)}} \sum_{\psi \bmod {16}}  \sum_{\substack{\varpi \odd \\ (\varpi, mk)=1}}   \frac{\Lambda_{K}(\varpi)\psi(\varpi)g_{j}(mk^3, \varpi) }{N(\varpi)^{s+w}}.
\end{split}
\end{align*}

  We now apply Proposition \ref{lemma:laundrylist} to the inner sum in the last expression above to deduce that
the last triple sum above and hence $A_{K,j}(s,w, z)$ is convergent in the region
\begin{equation}
\label{S2}
		S_{2} :=\left\{(s,w,z): \ \frac 32-\frac 16< \Re(s+w) \leq \frac 32, \ \Re\left(1-w+\frac 12 \left(s+w\right)\right) >\frac 74, \ \Re(z+s+w)>\frac 52 \right\}.
\end{equation}

 Moreover, in the region $S_{2, \varepsilon}$, we have by \eqref{Stirlingratio} that
\begin{align}
\label{Aswboundwlarge7}
\begin{split}
 \Big|A_{K,j}(s,w, z) \Big | \ll  (1+|w|)^{2-2\Re(w)+\varepsilon}.
\end{split}
\end{align}

  As the union of $S_1$ and $S_2$ is connected and the convex hull of $S_1$ and $S_2$ equals
\begin{align*}
\begin{split}
		S_3 :=& \left\{(s,w, z): \ \Re(s)> \frac 12, \ \Re(z)>\frac 12, \ \Re(s+z)> \frac 32, \ \Re(s+w)> \frac 32-\frac 16, \right. \\
& \hspace{0.1in} \left. \ \Re(s+3w+2z)> \frac 72, \ \Re(15s+13w+2z)> \frac {45}{2}, \ \Re\left( w+\frac {13}{11}\left(s-\frac 12 \right) \right)>1 \right\}, 
\end{split}
\end{align*}
Theorem \ref{Bochner} gives that $A_{K,j}(s,w,z)$ converges absolutely in the region $S_3$. \newline

 We set $S_{3,\varepsilon} = \{ (s,w,z)+\delta (1,1,1) : (s,w,z) \in S_3,  \ \Re(w) \leq 1 \}$ for any $\varepsilon>0$.
We then deduce from Proposition~\ref{Extending inequalities} that in the region $S_{3,\varepsilon} \cap \{ (s,w, z)| \Re(w) > \frac 12-\frac 1{11} \}$, we have
\begin{align}
\label{Aswboundslarge2}
\begin{split}
 |(s-1)A_{K,j}(s,w,z)| \ll &
\begin{cases}
(1+|s|)^{1+\varepsilon}(1+|w|)^{1+2/11+\varepsilon}|z|^{\varepsilon}, & \frac 12-\frac 1{11}< \Re(w) < \frac 12, \\
(1+|s|)^{1+\varepsilon}(1+|w|)^{\varepsilon}|z|^{\varepsilon}, & \Re(w) \geq \frac 12.
\end{cases}
\end{split}
\end{align}

 We now shift the contour of integration in \eqref{Integral for all characters} to $\Re(s)=E(\alpha, \beta)+\varepsilon$, encountering a simple pole at $s=1$ which contributes to the main term in \eqref{Asymptotic for ratios of all characters} using \eqref{Ress=1}.  The integral on the new line can be absorbed into the $O$-term in \eqref{Asymptotic for ratios of all characters} upon using \eqref{Aswboundslarge2} and the rapid decay of $\widehat{w}$. This completes the proof of Theorem \ref{Theorem for all characters}.

\vspace*{.5cm}

\noindent{\bf Acknowledgments.}  P. G. is supported in part by NSFC grant 12471003 and L. Z. by the FRG Grant PS71536 at the University of New South Wales.\newline

\noindent{\bf Data Availability Statement.} This manuscript has no associated data. \newline

\noindent{\bf Conflict of Interest Statement.} On behalf of all authors, the corresponding author states that there is no conflict of interest.

\bibliography{biblio}
\bibliographystyle{amsxport}

\end{document}